\documentclass[11pt, a4paper]{amsart}
\usepackage{amssymb}
\usepackage{amsmath}
\usepackage{amssymb}
\usepackage{amsthm}
\usepackage{bbm}
\usepackage{bm}
\usepackage{mathrsfs}
\usepackage{dsfont}
\usepackage{esint}
\usepackage{enumitem}
\setenumerate{label={\rm (\alph{*})}}
\usepackage[left=3.8cm,right=3.8cm,top=4cm,bottom=2.5cm]{geometry}
\usepackage[colorlinks=true, linktocpage=true, linkcolor=red!70!black, citecolor=green!50!black]{hyperref}
\usepackage{graphicx}
\usepackage{tikz}
\usetikzlibrary{matrix}

\numberwithin{equation}{section}

\newtheorem{theorem}{Theorem}[section]
\newtheorem{lemma}[theorem]{Lemma}
\newtheorem{proposition}[theorem]{Proposition}

\newtheorem{corollary}[theorem]{Corollary}

\newtheorem{remark}[theorem]{Remark}
\newtheorem{example}[theorem]{Example}

\numberwithin{equation}{section}

\normalsize

\renewcommand{\geq}{\geqslant}

\newcommand{\di}{\operatorname{div}}

\newcommand{\dif}{\operatorname{d}\!}

\newcommand{\R}{\mathbb{R}}

\newcommand{\uu}{\mathbf{u}}

\newcommand{\locc}{\operatorname{loc}}

\newcommand{\sym}{\operatorname{sym}}

\newcommand{\trace}{\operatorname{Tr}}

\newcommand{\ball}{\operatorname{B}}

\newcommand{\sobo}{\operatorname{W}}
\newcommand{\lebe}{\operatorname{L}}

\newcommand{\hold}{\operatorname{C}}

\newcommand{\curl}{\operatorname{curl}}
\renewcommand{\leq}{\leqslant}

\newcommand{\imag}{\operatorname{i}}

\newcommand{\sg}{\varepsilon}

\newcommand{\A}{\mathbb{A}}

\newcommand{\rsym}{\mathds{R}_{\sym}^{N \times n}}

\renewcommand{\rsym}{\mathbb{R}_{\operatorname{sym}}^{n\times n}}

\newcommand{\dashint}{\fint}

\def\uu{\textrm{u\kern -1.3ex\raisebox{0.6ex}{$^\circ$}}}
\usepackage{tikz-cd}
\begin{document}

\title[Partial Regularity and $\mathscr{A}$-Quasiconvexity]{$\mathscr{A}$-Quasiconvexity and Partial Regularity}
\subjclass[2010]{35J50, 35J93, 49J50}
\keywords{Calculus of variations, partial regularity, elliptic differential operators, $\mathscr{A}$-quasiconvexity, Korn inequalities, weighted Korn inequalities, Fourier multiplication operators.}
\date{\today}
\thanks{\emph{Funding.} Partially funded by the Deutsche Forschungsgemeinschaft (DFG, German
Research Foundation) {\sl via} project 211504053 - SFB 1060 and project
390685813 - HCM}

\author[S. Conti]{Sergio Conti}
\author[F. Gmeineder]{Franz Gmeineder}
\address{Sergio Conti: Universit\"{a}t Bonn, Institut f\"{u}r angewandte Mathematik, Endenicher Allee 60, 53115 Bonn, Germany}
\email{sergio.conti@uni-bonn.de}
\address{Franz Gmeineder: Mathematisches Institut der Universit\"{a}t Bonn, Endenicher Allee 60, 53115 Bonn, Germany}
\email{fgmeined@math.uni-bonn.de}
\maketitle

\date{\today}

\begin{abstract}
We establish the first partial regularity result for local minima of strongly $\mathscr{A}$-quasiconvex integrals in the case where the differential operator $\mathscr{A}$ possesses an elliptic potential $\A$. As the main ingredient, the proof works by reduction to the partial regularity for full gradient functionals. Specialising to particular differential operators, the results in this paper thereby equally yield novel partial regularity theorems in the cases of the trace-free symmetric gradient, the exterior derivative or the div-curl-operator. 
\end{abstract}

\tableofcontents
\section{Introduction}
A variety of minimisation problems which help to model properties of materials or fluids, so e.g. their elastic behaviour, can be stated in terms of non-convex energies. Such energies often do not only depend on the gradients of the quantities of interest (so e.g. the deformations of a material) but certain differential expressions; see \cite{BartnikIsenberg,CoFoIu,CoMuOr,Dain,FoMu,Friedrichs,FuchsSchirra,FuchsSeregin,MuPa} for discussions, among others in elasticity and general relativity. Typical examples thereof are given by the symmetric gradients $\sg(u):=\frac{1}{2}(Du+Du^{\top})$ or the trace-free symmetric gradients $\sg^{D}(u):=\sg(u)-\frac{1}{n}\di(u)\mathbbm{1}_{n}$ for maps $u\colon\R^{n}\supset\Omega\to\R^{n}$, and we refer the reader to Section~\ref{sec:examples} for more examples. In this paper we aim to give a unifying partial regularity theory for \emph{variational problems involving elliptic differential operators}, a theme that we describe now.

Let $V\cong\R^{N}$, $W\cong\R^{l}$ and $X\cong\R^{m}$ be real vector spaces and consider for linear maps $\mathbb{A}_{j}\colon V\to W$, $j\in\{1,...,n\}$, the vectorial differential operator
\begin{align}\label{eq:form}
\A u=\sum_{j=1}^{n}\A_{j}\partial_{j}u,\qquad u\colon\R^{n}\to V. 
\end{align}
In addition, we assume that $\A$ has \emph{constant rank}, meaning that the Fourier symbol of $\A$ has rank independent of phase space variables $\xi\in\R^{n}\setminus\{0\}$; cf.~\eqref{eq:symbols}\emph{ff}. below for the requisite background terminology. Given $1<p<\infty$, let $F\in\hold(W)$ be an integrand which satisfies for some constants $c_{1},c_{2},c_{3}>0$ the standard coerciveness and growth bounds
\begin{align}\label{eq:growth}
c_{1}|z|^{p}-c_{2}\leq F(z)\leq c_{3}(1+|z|^{p})\qquad\text{for all}\;z\in W. 
\end{align}
Given an open and bounded domain $\Omega\subset\R^{n}$, we consider the multiple integral 
\begin{align}\label{eq:functional}
\mathscr{F}[u;\omega]:=\int_{\omega}F(\A u)\dif x,\qquad \omega\subset\Omega,
\end{align}
which is well-defined for maps $u\colon\Omega\to V$  satisfying $\A u\in\lebe^{p}(\Omega;W)$. As usual, we say that $u\in\lebe_{\locc}^{p}(\Omega;V)$ is a \emph{local minimiser} provided $\A u\in\lebe_{\locc}^{p}(\Omega;W)$ and 
\begin{align}\label{eq:local}
\mathscr{F}[u;\omega]\leq\mathscr{F}[u+\varphi;\omega]\qquad\text{for all}\;\omega\subset\Omega\;\text{and}\;\varphi\in\hold_{c}^{\infty}(\omega;V). 
\end{align}
Essentially\footnote{In \cite{FoMu} only first order annihilators are considered, but the higher order case can be approached similarly; alternatively, this follows from the results of \textsc{Raita} \cite{Raita}.} by the foundational work of \textsc{Fonseca \& M\"{u}ller} \cite{FoMu}, the requisite lower semicontinuity for $\mathscr{F}$ (that is, if $u,u_{1},...\in\lebe_{\locc}^{1}(\Omega;V)$ satisfy $\A u_{j}\rightharpoonup\A u$ in $\lebe^{p}(\Omega;W)$, then $\mathscr{F}[u;\Omega]\leq \liminf_{j\to\infty}\mathscr{F}[u_{j};\Omega]$) is equivalent to $F$ being $\mathscr{A}$-quasiconvex for an annihilator $\mathscr{A}$ of $\A$. To describe the underlying terminology, a differential operator
\begin{align}\label{eq:form1}
\mathscr{A}u = \sum_{|\alpha|=k}\mathscr{A}_{\alpha}\partial^{\alpha}u,\qquad u\colon \R^{n}\to W
\end{align}
is called an \emph{annihilator} for $\A$ given by \eqref{eq:form} (and, conversely, $\A$ a \emph{potential} for $\mathscr{A}$) if and only if for each $\xi\in\R^{n}\setminus\{0\}$ the associated Fourier symbol complex 
\begin{align}\label{eq:complex}
V\stackrel{\A[\xi]}{\longrightarrow} W \stackrel{\mathscr{A}[\xi]}{\longrightarrow} Z
\end{align}
is exact at $W$: $\ker(\mathscr{A}[\xi])=\mathrm{im}(\A[\xi])$. Here, 
\begin{align}\label{eq:symbols}
\A[\xi]:=\sum_{j=1}^{n}\xi_{j}\A_{j},\;\;\;\mathscr{A}[\xi]:=\sum_{|\alpha|=k}\xi^{\alpha}\mathscr{A}_{\alpha},\qquad\xi=(\xi_{1},...,\xi_{n})\in\R^{n}
\end{align}
are the Fourier symbols of $\A$ or $\mathscr{A}$, respectively. In this situation, we say that $\A$ has \emph{constant rank} if 
\begin{align}
\dim(\mathrm{im}(\A[\xi]))\;\;\text{is independent of}\;\xi\in\R^{n}\setminus\{0\}, 
\end{align}
and this notion equally carries over to $\mathscr{A}$.

Following \cite{Dacorogna,FoMu}, given $\mathscr{A}$ of the form \eqref{eq:form1}, we call an integrand $F\in\hold(W)$ \emph{$\mathscr{A}$-quasiconvex} provided 
\begin{align}\label{eq:calAqc}
F(w)\leq \int_{(0,1)^{n}}F(w+\psi)\dif x
\end{align}
holds for all $w\in W$ and $\psi\in\hold^{\infty}(\mathbb{T}^{n};W)$ with $(\psi)_{(0,1)^{n}}=0$ and $\mathscr{A}\psi=0$, where $\mathbb{T}^{n}$ is the $n$-dimensional torus. In the situation where the symbol complex \eqref{eq:complex} is exact for each $\xi\in\R^{n}\setminus\{0\}$, it is not too difficult to show that the $\mathscr{A}$-quasiconvexity is equivalent to 
\begin{align}\label{eq:Aqc}
F(w)\leq \dashint_{U}F(w+\A\varphi)\dif x
\end{align}
for all open sets $U\subset\R^{n}$, $w\in W$ and $\varphi\in\hold_{c}^{\infty}(U;V)$ (cf.~\cite[Cor.~6, Lem.~8]{Raita}). If $\mathscr{A}$ has constant rank, then by a recent result due to \textsc{Raita} (\cite[Thm.~1]{Raita}, also see Lemma~\ref{lem:potentials} below), there always exists a potential $\A$ for $\mathscr{A}$ and so \eqref{eq:calAqc} and \eqref{eq:Aqc} are equivalent indeed. 

The aforementioned result due to \textsc{Fonseca \& M\"{u}ller} then easily implies the existence of local minima (or, e.g., global minima subject to certain Dirichlet constraints) by virtue of the direct method. We may hereafter inquire as to whether local minima share the by now well-known regularity features known from the usual full gradient theory for quasiconvex variational problems; cf. \cite{CFM98,DGK05,Evans,Giusti,Gm0,Gmeineder2,KristensenTaheri,Mingione1,Mingione2}  and the references therein for a non-exhaustive list. In this sense, the main focus of the present paper is on the \emph{regularity of local minima} subject to the $\mathscr{A}$-quasiconvexity of $F$.
\subsection{Partial regularity and main results}
 In general, since the minimisation of $\mathscr{F}[-;\Omega]$ given by \eqref{eq:functional} is a genuinely vectorial problem, a wealth of counterexamples \emph{even for the gradient $\A=D$} establishes that full H\"{o}lder regularity of minima cannot be expected; cf.~\cite{Beck,Giusti,Mingione1} for examples. A suitable substitute is then given by (the $\hold_{\locc}^{1,\alpha}$-)\emph{partial regularity}, meaning that there exists an open set $O\subset\Omega$ with $\mathscr{L}^{n}(\Omega\setminus O)=0$ such that $u\in\hold_{\locc}^{1,\alpha}(O;V)$ for all $0\leq\alpha<1$.

Based on the symbol complex~\eqref{eq:complex}, our first observation is that partial regularity of minima can be expected if and only if the potential $\A$ is \emph{elliptic}, a notion from the theory of overdetermined systems (cf.~\textsc{H\"{o}rmander}  \cite{Hoermander} and \textsc{Spencer} \cite{Spencer}). Namely, we say $\A$ given by \eqref{eq:form} is \emph{elliptic} if for each $\xi\in\R^{n}\setminus\{0\}$ 
\begin{align}\label{eq:ellipticity}
\text{the symbol map}\;\;\;\A[\xi]\colon V\to W\;\;\;\text{is injective}. 
\end{align}
Ellipticity of $\A$ is indeed equivalent to $\ker(\A)\subset\hold^{\infty}$ (cf.~Lemma~\ref{cor:smooth}), which in turn is required for the desired partial regularity.

We hereafter let $\A$ be an elliptic differential operator of the form \eqref{eq:form} and $\mathscr{A}$ an annihilator thereof. Adapting conditions which are by now routine for the full gradient situation, we hereafter suppose that the integrand $F\colon W\to\R$ satisfies the following set of hypotheses for some fixed $1<p<\infty$: 
\begin{enumerate}[label={(H\arabic{*})},start=1]
\item\label{item:H1} $F\in\hold^{2}(W)$. 
\item\label{item:H2} There exists $c>0$ such that $|F(z)|\leq c(1+|z|^{p})$ holds for all $z\in W$. 
\item\label{item:H3} $F$ is $p$-strongly $\mathscr{A}$-quasiconvex, meaning that there exists $\ell>0$ such that 
\begin{align*}
W\ni z \mapsto F(z)-\ell V_{p}(z):=F(z)-\ell \Big((1+|z|^{2})^{\frac{p}{2}}-1\Big)
\end{align*}
is $\mathscr{A}$-quasiconvex.
\end{enumerate}
Subject to these assumptions, the main result of the present paper is the following theorem:\begin{theorem}[Partial regularity]\label{thm:main1}
Let $\Omega\subset\R^{n}$ be open and bounded. Moreover, let $\A$ be a constant rank differential operator of the form \eqref{eq:form} and $\mathscr{A}$ be a constant rank differential operator of the form \eqref{eq:form1} such that \eqref{eq:complex} is exact at $W$ for every $\xi\in\R^{n}\setminus\{0\}$. Finally, suppose that $F\colon W\to\R$ satisfies \emph{\ref{item:H1}--\ref{item:H3}}. Then the following are equivalent: 
\begin{enumerate}
\item\label{item:THMelliptic} $\A$ is elliptic in the sense of \eqref{eq:ellipticity}.
\item\label{item:THMpr} Every local minimiser $u$ of $\mathscr{F}$ in the sense of \eqref{eq:local} is  $\hold_{\locc}^{1,\alpha}$-partially regular.
\end{enumerate}
\end{theorem}
One of the key aspects of the present paper is that the \emph{partial regularity statement of} Theorem~\ref{thm:main1} \emph{can be fully reduced to the corresponding full gradient theory}. In consequence, setting up a separate partial regularity proof is not required. Effectively, Theorem \ref{thm:main1} is a consequence of Korn-type inequalities for Orlicz functions and elliptic operators $\A$; these allow to set up a one-to-one correspondence between functionals of the form \eqref{eq:form} and full gradient functionals. This, in turn, allows us to access the by now well-understood regularity theory for the latter. The underlying reason for this reduction to work is that condition \ref{item:H3} expresses a certain coerciveness property for the associated full gradient functional, cf. Lemma~\ref{lem:coercive}, which becomes accessible by Korn-type inequalities. As a metaprinciple, if $\A$ is elliptic, then \emph{any regularity result available for local minima of signed strongly $p$-quasiconvex variational integrals directly inherits to the local minima of the corresponding functionals $\mathscr{F}$ given by \eqref{eq:functional}}. In more technical terms, the corresponding full gradient partial regularity theory must be available without growth bounds on the second derivatives. This, by now, is available in most of the settings addressed here; also see Theorem~\ref{thm:higherorder} for a result involving growth bounds on the second derivatives and $p\geq 2$. Also, in the special case where $\A$ is the symmetric gradient operator, an easy case of Theorem~\ref{thm:main1} has already been observed by the second author \cite{Gm1}; however, many of the arguments employed in \cite{Gm1} rely on background results that are well-understood for the symmetric gradient but far from clear in the unifying setting addressed here. 

Theorem \ref{thm:main1} thus displays a sample theorem. A discussion of other, more general scenarios and partial regularity is provided in Section~\ref{sec:examples} and many more interesting  generalisations such as Orlicz growth are conceivable. However, to keep our exposition at a reasonable length, the relevant generalisations stick to power growth assumptions throughout. 

Let us finally note that the above theorem rather adopts the potential (i.e., primarily working on $\A$ rather than $\mathscr{A}$) than the annihilator viewpoint, and in general the first order of $\A$ does not imply the first order of $\mathscr{A}$. Conversely, $\mathscr{A}$ need not have a first order potential, and so a higher order variant of Theorem~\ref{thm:main1} is required; see Theorem~\ref{thm:higherorder} for a corresponding result. 

\subsection{Structure of the paper}
In Sections~\ref{sec:prelims} and ~\ref{sec:diffops} we collect some background facts from harmonic analysis and establish auxiliary results for the treatment of differential operators as required later on. Section~\ref{sec:main}, along with Korn-type inequalities of independent interest, then is devoted to the proof of Theorem~\ref{thm:main1}; here we also briefly address the existence of minima, showing the naturality of conditions~\ref{item:H1}--\ref{item:H3}. The final Section~\ref{sec:examples} discusses examples and extensions to more general integrands.

\section{Preliminary results}\label{sec:prelims}
\subsection{Basic notation}
Throughout, $\Omega\subset\R^{n}$ is an open and bounded set, and for $x\in\R^{n}$ and $r>0$ we denote $\ball(x,r):=\{y\in\R^{n}\colon\;|x-y|<r\}$. Also, given a subset $U\subset\R^{n}$, we define $\mathrm{co}(U)$ to be the convex hull of $U$. All finite dimensional vector spaces $V$ are equipped with the usual euclidean norm $|\cdot|$ and inner product $\langle\cdot,\cdot\rangle$, and we denote the unit sphere in $V$ by $\mathbb{S}_{V}:=\{v\in V\colon\;|v|=1\}$; for ease of notation, we set $\mathbb{S}^{n-1}:=\{x\in\R^{n}\colon\,|x|=1\}$. Given $m\in\mathbb{N}$, the $V$-valued $m$-multilinear mappings on $\R^{n}$ are denoted $\odot^{m}(\R^{n};V)$. As usual, for a measurable set $\Omega\subset\R^{n}$ with $\mathscr{L}^{n}(\Omega)\in (0,\infty)$ and $f\in\lebe_{\locc}^{1}(\R^{n};V)$, we set
\begin{align*}
(f)_{\Omega}:=\dashint_{\Omega}f\dif x := \frac{1}{\mathscr{L}^{n}(\Omega)}\int_{\Omega}f\dif x.
\end{align*}
For $f\in\lebe^{1}(\R^{n};V)$, we use the following normalisation for the Fourier transform of $f$: 
\begin{align*}
\mathscr{F}f(\xi):=\widehat{f}(\xi):=\frac{1}{(2\pi)^{\frac{n}{2}}}\int_{\R^{n}}f(x)e^{-\imag\langle x,\xi\rangle}\dif x,\qquad \xi\in\R^{n}.
\end{align*} 
Finally, we write $a\lesssim b$ provided there exists a constant $C>0$ such that $a\leq Cb$ and $a\simeq b$ if $a\lesssim b$ and $b\lesssim a$; the underlying constants will be specified if required.

\subsection{Harmonic analysis and Orlicz integrands}
In this section we collect some auxiliary results from harmonic analysis which shall turn out instrumental for the partial regularity proof below. Let $1<q<\infty$. We say that $\omega\in\lebe_{\locc}^{1}(\R^{n})$ is an $A_{q}$-\emph{Muckenhoupt weight} (in formulas $\omega\in A_{q}$) if and only if $\omega> 0$ $\mathscr{L}^{n}$-a.e. in $\R^{n}$ and 
\begin{align}
A_{q}(\omega):=\sup_{Q}\left(\Big(\,\dashint_{Q}\omega\dif x\Big)\Big(\dashint_{Q}\omega^{-\frac{1}{q-1}}\dif x \Big)^{q-1} \right)<\infty, 
\end{align} 
the supremum ranging over all non-degenerate cubes $Q\subset\R^{n}$. We refer to $A_{q}(\omega)$ as the $A_{q}$\emph{-constant} of $\omega$. Given $\omega\in A_{q}$, we say that a constant $c=c(\omega)>0$ is \emph{$A_{q}$-consistent} if and only if it only depends on $A_{q}(\omega)$.

Let $\psi\colon\R_{\geq 0}\to \R_{\geq 0}$ be differentiable. We say that $\psi$ is an $N$-function provided $\psi(0)=0$ and its derivative $\psi'$ is right-continuous, non-decreasing together with
\begin{align}\label{eq:Phi1}
\psi'(0)=0,\;\;\;\psi'(t)>0\;\;\text{for}\;t>0\;\;\;\text{and}\;\lim_{t\to\infty}\psi'(t)=\infty. 
\end{align}
Given an $N$-function $\psi$, we say that $\psi$ is of class $\Delta_{2}$ if there exists $K>0$ such that $\psi(2t)\leq K\psi(t)$ for all $t\geq 0$, and define $\Delta_{2}(\psi)$ to be the infimum over all possible such constants. Analogously, we say that  $\psi$ is of class $\nabla_{2}$ provided the Fenchel conjugate $\psi^{*}(t):=\sup_{s\geq 0}(st-\psi(s))$ is of class $\Delta_{2}$, and we let $\nabla_{2}(\psi):=\Delta_{2}(\psi^{*})$. If $\psi$ satisfies both the $\Delta_{2}$- and the $\nabla_{2}$-condition, we say that $\psi$ is of class $\Delta_{2}\cap\nabla_{2}$. Each $N$-function $\psi$ gives rise to the Orlicz-Lebesgue space $\lebe^{\psi}(\R^{n};V)$, being defined as the linear space of all measurable $u\colon\R^{n}\to V$ with
\begin{align*}
\|u\|_{\lebe^{\psi}(\R^{n};V)}:=\inf\big\{\lambda>0\colon \int_{\R^{n}}\psi\left(\frac{|u|}{\lambda}\right)\dif x \leq 1\big\}<\infty.
\end{align*}
Clearly, if $\psi(t)=|t|^{p}$, then $\lebe^{\psi}=\lebe^{p}$. We next state a theorem of Mihlin-H\"{o}rmander type which substantially enters the proof of Theorem~\ref{thm:main1} below. 
\begin{lemma}[of Mihlin-H\"{o}rmander type]\label{lem:Mihlin}
Let $\Theta\in\hold^{\infty}(\R^{n}\setminus\{0\};\mathscr{L}(W;V))$ be homogeneous of degree zero and let $\psi\in\Delta_{2}\cap\nabla_{2}$ be an $N$-function. Then 
\begin{align*}
T_{\Theta}\colon \hold_{c}^{\infty}(\R^{n};W)\ni u \mapsto \mathscr{F}^{-1}\big[\Theta(\xi)\mathscr{F}u(\xi) \big](x),\qquad x\in\R^{n}, 
\end{align*}
extends to a bounded linear operator $T_{\Theta}\colon\lebe^{\psi}(\R^{n};W)\to\lebe^{\psi}(\R^{n};V)$. The operator norm $\|T_{\Theta}\|_{\lebe^{\psi}\to\lebe^{\psi}}$ only depends on $\Theta,\Delta_{2}(\psi)$ and $\nabla_{2}(\psi)$, and we have the following \emph{modular estimate}:
\begin{align}\label{eq:modular0}
\int_{\R^{n}}\psi(|T_{\Theta}u|)\dif x \leq c\int_{\R^{n}}\psi(|u|)\dif x\qquad\text{for all}\;u\in\lebe^{\psi}(\R^{n};W)
\end{align}
with a constant $c=c(\Theta,\Delta_{2}(\psi),\nabla_{2}(\psi))>0$. 
\end{lemma}
This lemma should be well-known to the experts, but we have been unable to trace it back to a precise reference. To give a quick argument, we recall the following extrapolation result due to \textsc{Cruz-Uribe} et al. \cite[Thm.~3.1]{CMP} in the version given by \textsc{Diening} et al. \cite[Prop.~6.1]{DRS10}:
\begin{lemma}\label{lem:CMP}
Let $1<q<\infty$ and suppose that $\mathcal{F}$ is a family of tuples $(f,g)\in\lebe_{\locc}^{1}(\R^{n})\times\lebe_{\locc}^{1}(\R^{n})$ such that for all $\omega\in A_{q}$ there holds 
\begin{align*}
\int_{\R^{n}}|f|^{q}\omega\dif x \leq K_{1}\int_{\R^{n}}|g|^{q}\omega\dif x \qquad\text{for all}\;(f,g)\in\mathcal{F}
\end{align*}
with an $A_{q}$-consistent constant $K_{1}>0$. Then for all $N$-functions $\varphi\in\Delta_{2}\cap\nabla_{2}$ there exists a constant $K_{2}=K_{2}(q,\Delta_{2}(\varphi),\nabla_{2}(\varphi))>0$ such that 
\begin{align}
\begin{split}
& \|f\|_{\lebe^{\varphi}(\R^{n})}\leq K_{1}K_{2}\|g\|_{\lebe^{\varphi}(\R^{n})},\\ 
& \int_{\R^{n}}\varphi(|f|)\dif x \leq K_{2}\int_{\R^{n}}\varphi(K_{1}|g|)\dif x
\end{split}
\end{align}
holds for all $(f,g)\in\mathcal{F}$. 
\end{lemma}
\begin{proof}[Proof of Lemma~\ref{lem:Mihlin}] 
We recall from \cite[Chpt.~4.5,~Thm.~4.3]{Duo00} that, if $m\in\hold^{\infty}(\R^{n}\setminus\{0\};\mathbb{C})$ is a function homogeneous of degree zero, then the multiplier operator $\mathscr{S}(\R^{n})\ni f \mapsto T_{m}f:=\mathscr{F}^{-1}(m\widehat{f})$ can be represented as 
\begin{align}\label{eq:Tmsplit}
T_{m}f=T_{m}^{(1)}f + T_{m}^{(2)}f := af+\mathrm{p.v.}\Big(\frac{1}{|\cdot|^{n}}\Xi(\tfrac{\cdot}{|\cdot|})\Big)*f ,\qquad f\in\mathscr{S}(\R^{n})
\end{align}
for some $a\in\mathbb{C}$ and $\Xi\in\hold^{\infty}(\mathbb{S}^{n-1})$ with zero average, where $\mathbb{S}^{n-1}$ is the $(n-1)$-dimensional unit sphere and $\mathrm{p.v.}$ denotes the Cauchy principal value. In the standard terminology of harmonic analysis, $T_{m}^{(2)}$ then is a Calder\'{o}n-Zygmund operator (cf.~\cite[Def.~5.11]{Duo00}). In consequence, by the results of \cite{Hyt12}, there exists a constant $c=c(m)>0$ such that 
\begin{align*}
\|T_{m}\|_{\lebe_{\omega}^{2}(\R^{n})\to\lebe_{\omega}^{2}(\R^{n})}\leq c(m)A_{2}(\omega)
\end{align*}
for all $\omega\in A_{2}$. Now let $\psi\in\Delta_{2}\cap\nabla_{2}$. Since the $(\lebe_{\omega}^{2}\to\lebe_{\omega}^{2})$-operator norm of $T_{m}$ thence is $A_{2}$-consistent, Lemma~\ref{lem:CMP} with $q=2$ yields the existence of some $K_{1}=K_{1}(m)>0$ and $K_{2}=K_{2}(\Delta_{2}(\psi),\nabla_{2}(\psi))>0$ such that $\|T_{m}f\|_{\lebe^{\psi}(\R^{n})}\leq K_{1}K_{2}\|f\|_{\lebe^{\psi}(\R^{n})}$ and 
\begin{align*}
\int_{\R^{n}}\psi(|T_{m}f|)\dif x \leq K_{2}\int_{\R^{n}}\psi(K_{1}|f|)\dif x
\end{align*}
for all $f\in\lebe^{\psi}(\R^{n})$. Identifying $V\cong\R^{N}$ and $W\cong\R^{l}$, we may write $\Theta=(\Theta_{ij})_{1\leq i \leq N,\,1\leq j \leq l}$ and $u=(u_{1},...,u_{l})$. Then $(T_{\Theta}u)_{i}=\sum_{j=1}^{l}\mathscr{F}^{-1}[\Theta_{ij}\widehat{u}_{j}]$ for $i=1,...,N$. Applying the above  to $m=\Theta_{ij}$ and $f=u_{j}$, we consequently obtain \eqref{eq:modular0} because of $\psi\in\Delta_{2}\cap\nabla_{2}$.  The proof is complete. 
\end{proof}
In the above proof, the requisite boundedness of $T_{m}$ also follows by \cite[Thm.~7.11]{Duo00}, but \cite{Hyt12} gives a particularly transparent tracking of the dependencies of the constants.

Now let $\psi\in\hold^{1}([0,\infty))\cap\hold^{2}((0,\infty))$ be an $N$-function which satisfies $\psi'(t)\simeq t\psi''(t)$ uniformly in $t>0$. Following \cite{DE08}, we then define for $a\geq 0$ the corresponding \emph{shifted $N$-function} $\psi_{a}\colon\R_{\geq 0}\to \R_{\geq 0}$ by 
\begin{align}\label{eq:shifted}
\psi_{a}(t):=\int_{0}^{t}\frac{\psi'(a+s)}{a+s}s\dif s,\qquad t\geq 0.
\end{align}
We then record from \cite[Lem.~23]{DE08} and \cite[Def.~2]{DLSV} (also see \cite[Sec.~B]{DFTW}) the following background facts: 
\begin{enumerate}[label={(F\arabic{*})},start=1]
\item\label{item:F1} There exists $c=c(\Delta_{2}(\psi),\nabla_{2}(\psi))>0$ such that 
\begin{align*}
\frac{1}{c}\psi_{a}(t)\leq \psi''(a+t)t^{2} \leq c\psi_{a}(t)\qquad\text{for all}\;a,t\geq 0.
\end{align*}
\item\label{item:F2} There holds $\psi_{a}\in\Delta_{2}\cap\nabla_{2}$ for all $a\geq 0$ and we have
\begin{align*}
\Delta_{2}(\psi_{a}) \simeq \Delta_{2}(\psi)\;\;\;\text{and}\;\;\;\nabla_{2}(\psi_{a})\simeq \nabla_{2}(\psi)
\end{align*}
uniformly in $a$. Thus, the family $(\psi_{a})_{a\geq 0}$ satisfies the $\Delta_{2}$- and $\nabla_{2}$-conditions uniformly in $a\geq 0$. 
\end{enumerate}
Finally, a comparison lemma; for this, we define the auxiliary map $V_{p}$ on $\R^{m}$ by 
\begin{align}\label{eq:Vpfunction}
V_{p}(z):=(1+|z|^{2})^{\frac{p}{2}}-1,\qquad z\in\R^{m}.
\end{align}
\begin{lemma}[{\cite[Lem.~2.4]{DFLM}, \cite[Sec.~6.2, (6.5)ff.]{Gm1}}]\label{lem:compare}
Let $1<p<\infty$ and $m\in\mathbb{N}$. 
\begin{enumerate}
\item\label{item:compare1} Define $V_{p}(z)$ for $z\in\R^{m}$ by \eqref{eq:Vpfunction}. Then there exists $0<\theta_{p}<\infty$ such that for all $z,w\in\R^{m}$ there holds
\begin{align}
\begin{split}
\frac{1}{\theta_{p}}(1+|z|^{2}+|w|^{2})^{\frac{p-2}{2}}|w|^{2} & \leq V_{p}(z+w) -V_{p}(z)-\langle V'_{p}(z),w\rangle \\ & \leq \theta_{p}(1+|z|^{2}+|w|^{2})^{\frac{p-2}{2}}|w|^{2}.
\end{split}
\end{align}
\item\label{item:compare2} Let $1<p<2$ and define $\Psi\colon [0,\infty)\to[0,\infty)$ by $\Psi(t):=(1+t)^{p-2}t^{2}$. Then there exists a constant $c=c(p)>0$ such that for all $a,t\geq 0$ there holds 
\begin{align}\label{eq:PsiBounds}
\begin{split}
\frac{1}{c}\Psi''(a+t)&\leq (1+a^{2}+t^{2})^{\frac{p-2}{2}} \leq c \Psi''(a+t), \\ 
\frac{1}{c}\Psi''(t)t &\leq \Psi'(t) \leq c\Psi''(t)t.
\end{split}
\end{align} 
\end{enumerate}
\end{lemma}
%
\section{Vectorial differential operators}\label{sec:diffops}
In the sequel, let $\A$ be a differential operator of the form \eqref{eq:form}. For future reference, we further set with the operator norm $|\cdot|$ on $\mathscr{L}(V;W)$
\begin{align}\label{eq:normofA}
\|\A\|:=\sum_{j=1}^{n}|\A_{j}|. 
\end{align}
We define the set $\mathscr{C}(\A)$ of \emph{pure $\A$-tensors} $v\otimes_{\A}\xi$ as the collection of elements
\begin{align}\label{eq:puretensors}
v\otimes_{\A}\xi := \A[\xi]v = \sum_{j=1}^{n}\xi_{j}\A_{j}v,\qquad\text{for}\;v\in V,\;\xi=(\xi_{1},...,\xi_{n})\in\R^{n}. 
\end{align}
Then $\mathscr{C}(\A)\subset W$, and we define $\mathscr{R}(\A)$ to be the linear hull of $\mathscr{C}(\A)$. The space $\mathscr{R}(\A)$ is (up to isomorphy) the smallest space in which $\A v(x)$ takes values when $v$ ranges over $\hold^{\infty}(\R^{n};V)$ and $x$ ranges over $\R^{n}$, see \cite[Sec.~5]{BDG}. Indeed, in the definition of $\A$ (cf.~\eqref{eq:form}) we might replace $W$ by any other $W'$ such that $W\hookrightarrow W'$, neither destroying the constant rank nor ellipticity properties of $\A$. For example, if 
\begin{itemize}
\item $\A u=\sg(u)=\frac{1}{2}(D+D^{\top})$ is the symmetric gradient, we may take $W=\R^{n\times n}$ and in this case, $\mathscr{R}(\A)=\R_{\sym}^{n\times n}$, the symmetric $(n\times n)$-matrices. 
\item $\A u=\sg^{D}(u)=\sg(u)-\frac{1}{n}\di(u)\mathbbm{1}_{n}$ is the trace-free symmetric gradient, we may take $W=\R^{n\times n}$ or $W=\R_{\sym}^{n\times n}$, and in this case, $\mathscr{R}(\A)=\R_{\sym,\mathrm{tf}}^{n\times n}$, the symmetric, trace-free $(n\times n)$-matrices. 
\end{itemize}
The benefit of passing to $\mathscr{R}(\A)$ is illustrated in Example~\ref{ex:duplicate} below. Working with general finite dimensional vector spaces $V$ and $W$ has some advantages, letting us e.g. deal with the exterior derivatives, cf. Section~\ref{sec:examples1}, Example~\ref{item:EX3}. We now argue that we may \emph{assume} $\mathscr{R}(\A)=W\subset\R^{N\times n}$ and that $F$ as in Theorem~\ref{thm:main1} can be tacitly supposed to be defined on a subset of the real $(N\times n)$-matrices. With our main results being established in this situation, it is then a somewhat lenghty yet elementary identification procedure between finite dimensional vector spaces to conclude Theorem~\ref{thm:main1} in the general case too; this is explained carefully in the Appendix, Section~\ref{sec:appendix}.


Denote $\Pi_{\A}\colon W\to\mathscr{R}(\A)$ the orthogonal projection onto $\mathscr{R}(\A)$. As to variational integrals \eqref{eq:functional} with $F\colon W\to\R$, we then have 
\begin{align}\label{eq:essentialreduction}
F(\A v(x))=F(\Pi_{\A}(\A v(x)))\qquad\text{for all}\;v\in\hold_{c}^{1}(\R^{n};V),\;x\in\R^{n}. 
\end{align}
Recalling that $V\cong\R^{N}$, we note that 
\begin{align}\label{eq:dimnote}
\dim(\mathscr{R}(\A))\leq \dim(V)n = Nn, 
\end{align}
which follows from the fact that, if $\{\mathbf{v}_{1},...,\mathbf{v}_{N}\}$ is a basis for $V$, then $\{\mathbf{v}_{i}\otimes_{\A}e_{j}\colon\;i=1,...,N,\,j=1,...,n\}$ (with $e_{j}$ being the $j$-th standard unit vector of $\R^{n}$) spans $\mathscr{R}(\A)$. As a main consequence of \eqref{eq:dimnote}, we may now assume that 
\begin{center}
\fbox{$V=\R^{N}$ and $W=\mathscr{R}(\A)\subset\R^{N\times n}$.}
\end{center}
Similarly as in \cite{Gmeineder2a}, we may thus write for $v=(v_{1},...,v_{N})\colon\R^{n}\to\R^{N}$
\begin{align}\label{eq:AOpdef}
\begin{split}
\A v(x) & := \pi_{\A}(\nabla v)(x) \\ & := \left(\begin{matrix} \sum_{i=1}^{n}\sum_{j=1}^{N}a_{1,1}^{i,j}\partial_{i}v_{j}(x) & \hdots & \sum_{i=1}^{n}\sum_{j=1}^{N}a_{1,n}^{i,j}\partial_{i}v_{j}(x) \\ \vdots & \ddots & \vdots \\ \sum_{i=1}^{n}\sum_{j=1}^{N}a_{N,1}^{i,j}\partial_{i}v_{j}(x) & \hdots & \sum_{i=1}^{n}\sum_{j=1}^{N}a_{N,n}^{i,j}\partial_{i}v_{j}(x) \end{matrix}\right),
\end{split}
\end{align}
where the linear map $\pi_{\A}\colon\R^{N\times n}\to\R^{N\times n}$ is defined in the obvious manner; here, $a_{\mu,\nu}^{i,j}\in\R$ for all $j,\mu\in\{1,...,N\}$ and $i,\nu\in\{1,...,n\}$.

\begin{example}[Duplicating elliptic operators]\label{ex:duplicate}
Let $\A$ an elliptic differential operator of the form \eqref{eq:form} with $V=\R^{N}$ and $W=\mathscr{R}(\A)\subset\R^{N\times n}$. Then for any $b_{1}\neq 0$ and $b_{2},...,b_{m}\in\R$ the $\mathscr{R}(\A)^{m}$-valued operator $\mathbb{B}$ defined by 
\begin{align*}
\mathbb{B}u = (b_{1}\A u, b_{2}\A u, ... ,b_{m}\A u)
\end{align*}
remains an elliptic first order differential operator. Still, $\mathbb{B}u$ is completely determined by the \emph{essential first component} $\alpha_{1}\mathbb{A}u$, and this is precisely reflected by passing to the essential range $\mathscr{R}(\A)$; note that $\dim(\mathscr{R}(\mathbb{B}))=\dim(\mathscr{R}(\A))$.
\end{example}
We conclude this preliminary section with the following background result, linking the $\mathscr{A}$-free and the $\A$-differential framework:
\begin{lemma}[{\textsc{Van Schaftingen} \cite[Prop.~4.2]{Va11}, \textsc{Raita} \cite[Thm.~1.1]{Raita}}]\label{lem:potentials}
The following hold: 
\begin{enumerate}
\item\label{item:VScompare} Let $\A$ be an elliptic differential operator of the form \eqref{eq:form}. Then there exists $k\in\mathbb{N}$, a real, finite dimensional vector space $Z$ and a $k$-th order $Z$-valued constant-rank  differential operator $\mathscr{A}$ of the form \eqref{eq:form1} such that \eqref{eq:complex} is exact at $W$ for any $\xi\in\R^{n}\setminus\{0\}$. 
\item\label{item:Rcompare} Let $\mathscr{A}$ be a constant-rank differential operator of the form \eqref{eq:form1}. Then there exists a real, finite dimensional vector space $V$, $l\in\mathbb{N}$ and a differential operator $\A=\sum_{|\alpha|=l}\A_{\alpha}\partial^{\alpha}$ with $\A_{\alpha}\in\mathscr{L}(V;W)$ such that \eqref{eq:complex} is exact at $W$ for any $\xi\in\R^{n}\setminus\{0\}$. 
\end{enumerate}
\end{lemma}
Note that even $\A$ might have first order, $\mathscr{A}$ might have higher order (which is e.g. the case for $\A=\sg$ and $\mathscr{A}=\curl\curl$, called the \textsc{Saint-Venant} compatibility conditions).
\section{Elliptic potentials and the proof of Theorem~\ref{thm:main1}}\label{sec:main}
\subsection{A family of Korn-type inequalities}
Before we embark on the proof of Theorem~\ref{thm:main1}, we establish a family of Korn-type inequalities which enter our subsequent arguments in an instrumental way. Since it might be of independent interest, we state the result in a slightly sharper  and more general way than it is actually required below: 
\begin{proposition}[of Korn-type]\label{prop:Korn}
Let $\psi\in \Delta_{2}\cap\nabla_{2}$ and suppose that $\mathbb{A}$ is a differential operator of the form \eqref{eq:form}.
Then the following are equivalent: 
\begin{enumerate}
\item\label{item:Korn1} $\mathbb{A}$ is elliptic. 
\item\label{item:Korn2} There exists a constant $c>0$ depending only on $\mathbb{A}$, $\Delta_{2}(\psi)$ and $\nabla_{2}(\psi)$ such that for all $u\in\hold_{c}^{\infty}(\R^{n};V)$ there holds
\begin{align}\label{eq:modular}
\int_{\R^{n}}\psi(|Du|)\dif x \leq c\int_{\R^{n}}\psi(|\A u|)\dif x. 
\end{align}
\end{enumerate}
\end{proposition}
\begin{proof}
Ad '\ref{item:Korn1}$\Rightarrow$\ref{item:Korn2}'. The proof is a consequence of the Mihlin multiplier theorem in the version as given in Lemma~\ref{lem:Mihlin}. By ellipticity of $\mathbb{A}$, for each $\xi\in\R^{n}\setminus\{0\}$ the Fourier symbol $\mathbb{A}[\xi]\colon V\to W$ is injective and hence for each such $\xi$, $\A^{*}[\xi]\A[\xi]\colon V\to V$ is bijective. For $j\in\{1,...,n\}$ define an operator for $w\in\hold_{c}^{\infty}(\R^{n};W)$ by 
\begin{align}\label{eq:inverse}
\Phi_{j,\A}(w)(x) = \mathscr{F}_{\xi\mapsto x}^{-1}\big[\xi_{j}(\mathbb{A}^{*}[\xi]\mathbb{A}[\xi])^{-1}\mathbb{A}^{*}[\xi]\mathscr{F}w(\xi)\big],\qquad x\in\R^{n}.
\end{align}
Put $\Theta_{j,\mathbb{A}}(\xi):=\xi_{j}(\mathbb{A}^{*}[\xi]\mathbb{A}[\xi])^{-1}\mathbb{A}^{*}[\xi]$ for $\xi\in\R^{n}\setminus\{0\}$. Clearly,  the multiplier $\Theta_{j,\mathbb{A}}\in\hold^{\infty}(\R^{n}\setminus\{0\};\mathscr{L}(W;V))$ is homogeneous of degree zero, and with the notation of Lemma~\ref{lem:Mihlin}, $\Phi_{j,\mathbb{A}}=T_{\Theta_{j,\mathbb{A}}}$. Hence, by Lemma~\ref{lem:Mihlin} and $\psi\in\Delta_{2}\cap\nabla_{2}$, $\Phi_{j,\A}$ extends to a bounded linear operator $\lebe^{\psi}(\R^{n};W)\to\lebe^{\psi}(\R^{n};V)$ which also satisfies the modular estimate \eqref{eq:modular} by virtue of \eqref{eq:modular0}. Since $\Phi_{j,\A}(\A u)=\partial_{j}u$ everywhere for all $u\in\hold_{c}^{\infty}(\R^{n};V)$, \eqref{eq:modular} follows at once.

Ad  '\ref{item:Korn2}$\Rightarrow$\ref{item:Korn1}'. Suppose that $\mathbb{A}$ is not elliptic. Then there exists $\xi'\in\mathbb{S}^{n-1}$ such that $\A[\xi']v=0$ for some $v\in V$ with $|v|=1$. We choose $\mathbf{e}_{2},...,\mathbf{e}_{n}\in\R^{n}$ such that $\{\xi',\mathbf{e}_{2},...,\mathbf{e}_{n}\}$ is an orthonormal basis for $\R^{n}$ and define 
\begin{align*}
D:=\mathrm{co}\{\pm\xi'\pm\mathbf{e}_{2}\pm...\pm\mathbf{e}_{n}\}.
\end{align*}
Then $D\subset\ball(0,\sqrt{n})$, and we choose $\rho\in\hold_{c}^{\infty}(\R^{n};[0,1])$ with 
\begin{align}
\mathbbm{1}_{\ball(0,\sqrt{n})}\leq\rho\leq\mathbbm{1}_{\ball(0,R)},\;\;\text{where}\;R=2\sqrt{n}\left(1+\frac{1}{\sqrt{n}\min\{\tfrac{1}{\|A\|},1\}}\right)
\end{align}
and, by our choice of $R$, $|\nabla\rho|\leq \min\{\frac{1}{\|A\|},1\}$; here, $\|A\|$ is given by \eqref{eq:normofA}. Let $(h_{i})\subset\hold_{c}^{\infty}((-1,1))$ be such that $\|\psi(|h_{i}|)\|_{\lebe^{1}(\R)}\to 0$ and $\|\psi(|h'_{i}|)\|_{\lebe^{1}(\R)}\to\infty$  as $i\to\infty$. We define the plane waves  $u_{i}(x):=\rho(x)h_{i}(\langle x,\xi'\rangle)v$ so that $u_{i}\in\hold_{c}^{\infty}(\R^{n};V)$ and, since $\A[\xi']v=0$, 
\begin{align}\label{eq:derivatives}
\begin{split}
&\A u_{i}(x)=h_{i}(\langle x,\xi'\rangle)v\otimes_{\A}\nabla\rho(x),\\ & Du_{i}(x) = \rho(x)h'_{i}(\langle x,\xi'\rangle)v\otimes \xi' + h_{i}(\langle x,\xi'\rangle)v\otimes\nabla\rho.
\end{split}
\end{align}
In conclusion,
\begin{align*}
\int_{\R^{n}}\psi(|\A u_{i}|)\dif x & \stackrel{\eqref{eq:derivatives}}{=} \int_{\R^{n}}\psi(|h_{i}(\langle x,\xi'\rangle)v\otimes_{\A}\nabla\rho|)\dif x\\ & \!\!\!\!\!\!\!\!\stackrel{|\nabla\rho|\leq\|A\|^{-1}}{\leq} \int_{\ball(0,R)}\psi(|h_{i}(\langle x,\xi'\rangle)|)\dif x \\ & \;\,\leq  \int_{RD}\psi(|h_{i}(\langle x,\xi'\rangle)|)\dif x   \leq (2R)^{n-1}\int_{-R}^{R}\psi(|h_{i}(t)|)\dif t\to 0, 
\end{align*}
having used a change of variables of the usual euclidean basis to $\{\xi',\mathbf{e}_{2},...,\mathbf{e}_{n}\}$ and Fubini's theorem in the last inequality. Similarly, we obtain that 
\begin{align*}
\|\psi(|h_{i}(\langle\cdot,\xi'\rangle)v\otimes\nabla\rho|)\|_{\lebe^{1}(\R^{n})}\to 0.
\end{align*} 
On the other hand, with $D$ as above, 
\begin{align*}
\int_{\R^{n}}\psi(|\rho(x)h'_{i}(\langle x,\xi'\rangle) v\otimes\xi'|)\dif x & \geq \int_{D}\psi(|h'_{i}(\langle x,\xi'\rangle)|\,|v\otimes\xi'|)\dif x \\ 
& = 2^{n-1}\int_{-1}^{1}\psi(|h'_{i}(t)|)\dif t \to \infty.
\end{align*}
Because of $\eqref{eq:derivatives}_{2}$, we deduce by convexity and the $\Delta_{2}$-condition of $\psi$ that necessarily $\|\psi(|Du_{j}|)\|_{\lebe^{1}(\R^{n})}\to\infty$, creating a contradiction to \eqref{eq:modular}. The proof is complete. 
\end{proof}
Proposition~\ref{prop:Korn} generalises earlier works (cf.~\cite{AcerbiMingione,BD}) in a somewhat optimal way. Even though it is not required for our main proof below, we believe that it is possible to strengthen the previous proposition for $N$-functions $\psi$ and differential operators $\A$ of the form \eqref{eq:form} as follows. Namely, adapting the approach to counterexamples to $\lebe^{1}$-estimates in \cite{CFM05} as pursued in \cite{BD} in the case of the symmetric gradients, validity of \ref{item:Korn2} should be equivalent to 
\begin{enumerate}
\item[(a')] $\A$ is elliptic and at least one of the following holds: Either there exists a linear map $T\in\mathscr{L}(W;V\times\R^{n})$ such that $Du=T\A u$ for all $u\in\hold_{c}^{\infty}(\R^{n};V)$ or $\psi\in\Delta_{2}\cap\nabla_{2}$. 
\end{enumerate}
This would be in line with the results of \textsc{Kirchheim \& Kristensen} \cite{KiKr16} regarding the trivialisation of $\lebe^{1}$-estimates within the framework of Orlicz functions, and we intend to pursue this question in the future. Let us, however, remark that the above proof yields the following by-product:
\begin{corollary}\label{cor:Korn1}
Let $1<q<\infty$ and $\omega\in A_{q}$. If $\A$ is an elliptic differential operator of the form \eqref{eq:form}, then there exists a constant $c=c(q,\A,A_{q}(\omega))>0$ such that 
\begin{align*}
\int_{\R^{n}}|Du|^{q}\omega\dif x \leq c\int_{\R^{n}}|\A u|^{q}\omega\dif x \qquad\text{for all}\;u\in\hold_{c}^{\infty}(\R^{n};V). 
\end{align*}
\end{corollary} 
Indeed, since Calder\'{o}n-Zygmund operators are bounded on $\lebe_{\omega}^{q}$ if $\omega\in A_{q}$ (cf.~\cite[Thm.~7.11]{Duo00}), Corollary~\ref{cor:Korn1} follows from \eqref{eq:inverse} and \eqref{eq:Tmsplit} \emph{ff.}.

Standard elliptic regularity theory and the plane wave construction underlying Proposition~\ref{prop:Korn} moreover imply the following
\begin{corollary}\label{cor:smooth}
Let $\Omega\subset\R^{n}$ be open and let $\A$ be a constant rank differential operator of the form \eqref{eq:form}. If $\A$ is elliptic, then $\ker(\A;\Omega)\cap\lebe_{\locc}^{1}(\Omega;V)\subset\hold^{\infty}(\Omega;V)$. Conversely, if $\A$ is not elliptic, then there exists $v\in\ker(\A;\Omega)\cap\lebe_{\locc}^{1}(\Omega;V)$ such that $v\notin\hold(\omega;V)$ for any open subset $\omega\subset\Omega$. 
\end{corollary}
\begin{proof} 
If $u\in\mathscr{D}'(\Omega;V)$ solves $\A u=0$, then $\A^{*}\A u=0$, and so the $\hold^{\infty}$-regularity of $u$ follows from by now classical results for second order elliptic systems. For the second part, take a direction $\xi'\in\R^{n}\setminus\{0\}$ and a vector $v\in V$ as in the proof of Proposition~\ref{prop:Korn}, direction '\ref{item:Korn2}$\Rightarrow$\ref{item:Korn1}'. Then for any $h\in\lebe_{\locc}^{1}(\R)$ which is not continuous in any neighbourhood of any $x_{0}\in\R$, we put $v(x):=h(\langle x,\xi'\rangle)v$. Then $\A v = 0$ in $\mathscr{D}'(\omega;W)$ but $v\notin\hold(\omega;V)$ for any open set $\omega\subset\R^{n}$.  
\end{proof} 
\subsection{Proof of Theorem~\ref{thm:main1}}
We now come to the proof of Theorem~\ref{thm:main1}. Let us note that it is only the direction \ref{item:THMelliptic}$\Rightarrow$\ref{item:THMpr} which actually requires an argument: Namely, if $\A$ is not elliptic and $u\in\lebe_{\locc}^{1}(\Omega;V)$ is a local minimiser for $\mathscr{F}$ given by \eqref{eq:functional}, we utilise Corollary~\ref{cor:smooth} to find $v\in\ker(\A;\Omega)$ with $v\notin\hold(\omega;W)$ for all open subsets $\omega\subset\Omega$. If there is no $\hold_{\locc}^{1,\alpha}$-partially regular local minimiser, we are done. If, instead, there does exist a $\hold_{\locc}^{1,\alpha}$-partially regular minimiser, $u+v$ is clearly not $\hold_{\locc}^{1,\alpha}$-regular on any open subset on which $u$ is. 

Similarly as in \cite[(3.1)ff.]{Gm1}, we record that the $p$-strong $\mathscr{A}$-quasiconvexity of $F$ as asserted in \ref{item:H3} is equivalent to the existence of a constant $\nu>0$ such that 
\begin{align}\label{eq:pstrongEQUIV}
\nu\int_{(0,1)^{n}}(1+|z|^{2}+|\A\varphi|^{2})^{\frac{p-2}{2}}|\A\varphi|^{2}\dif x \leq \int_{(0,1)^{n}}F(z+\A\varphi)-F(z)\dif x
\end{align}
holds for all $z\in\mathscr{R}(\A)$ and $\varphi\in\hold_{c}^{1}((0,1)^{n};V)$. This is a consequence of Lemma~\ref{lem:compare}~\ref{item:compare1} and Lemma~\ref{lem:potentials}~\ref{item:VScompare}. 

Toward the proof of Theorem~\ref{thm:main1}, we recall from the full gradient regularity theory that a continuous integrand $H\colon\R^{N\times n}\to\R$ that satisfies 
\begin{align}\label{eq:Hgrowth}
|H(z)|\leq L(1+|z|^{p})\qquad\text{for all}\;z\in\R^{N\times n}
\end{align}
for some $L>0$ is called \emph{$p$-strongly quasiconvex} if there exists $\ell>0$ such that $H-\ell V_{p}$ is quasiconvex. Then, if $p\geq 2$, we have 
\begin{align*}
t^{2}+t^{p} =t^{2}+t^{p-2}t^{2} \leq t^{2}+(1+t^{2})^{\frac{p-2}{2}}t^{2} \leq 2(1+t^{2})^{\frac{p-2}{2}}t^{2}
\end{align*}
for all $t\geq 0$. Thus, if $p\geq 2$, the $p$-strong quasiconvexity of $H$ implies that 
\begin{align}\label{eq:weakerStrongQC}
\mu\int_{(0,1)^{n}}|D\varphi|^{2}+|D\varphi|^{p}\dif x & \leq \int_{(0,1)^{n}}H(z+D\varphi)-H(z)\dif x
\end{align}
with $\mu=\frac{\ell}{2\theta_{p}}$ for all $\varphi\in\hold_{c}^{1}((0,1)^{n};\R^{N})$,  $z\in\R^{N\times n}$ and the constant $\theta_{p}>0$ from Lemma~\ref{lem:compare}~\ref{item:compare1}. We then rely on the following background result by \textsc{Acerbi \& Fusco} \cite{AF1} and \textsc{Carozza, Fusco \& Mingione} \cite{CFM98}: 
\begin{proposition}[{\cite[Thm.~II.1]{AF1}, \cite[Thm.~3.2]{CFM98}}]\label{prop:CFM}
Let $1<p<\infty$ and let $\Omega\subset\R^{n}$ be open. Suppose that $H\in\hold^{2}(\R^{N\times n})$ satisfies \eqref{eq:Hgrowth} for all $z\in\R^{N\times n}$ and, 
\begin{enumerate}
\item\label{item:reductionreg1} if $p\geq 2$, \eqref{eq:weakerStrongQC} holds for some $\mu>0$, all $z\in\R^{N\times n}$ and all $\varphi\in\hold_{c}^{\infty}((0,1)^{n};\R^{N})$,
\item\label{item:reductionreg2} if $1<p<2$, $H$ is $p$-strongly quasiconvex. 
\end{enumerate}
Then any local minimiser $u\in\sobo_{\locc}^{1,p}(\Omega;\R^{N})$ of the integral functional 
\begin{align*}
v\mapsto \int H(Dv)\dif x
\end{align*}
is $\hold_{\locc}^{1,\alpha}$-partially regular on $\Omega$.
\end{proposition}
We now come to the 
\begin{proof}[Proof of Theorem~\ref{thm:main1}, Case $2\leq p <\infty$]
By \eqref{eq:essentialreduction} \emph{ff.}, we may assume that $V=\R^{N}$ and $W=\mathscr{R}(\A)\subset\R^{N\times n}$. Adopting the terminology of \eqref{eq:AOpdef}, we put
\begin{align}\label{eq:Gdef}
G\colon \R^{N\times n}\ni z \mapsto F(\pi_{\A}(z))\in\R.
\end{align}
Since $F\in\hold^{2}(\mathscr{R}(\A))$ by \ref{item:H1} and the operator $\pi_{\A}\colon\R^{N\times n}\to\R^{N\times n}$ is linear, $G\in\hold^{2}(\R^{N\times n})$. Moreover, by \ref{item:H2} we have $|G(z)|\leq C(1+|z|^{p})$ for all $z\in\R^{N\times n}$ and some constant $C>0$. By \eqref{eq:pstrongEQUIV} and analogously to \eqref{eq:weakerStrongQC},
$p\geq 2$ implies that
\begin{align}\label{eq:pstrongEQUIV1}
\frac{\ell}{2\theta_{p}}\int_{(0,1)^{n}}|\A\varphi|^{2}+|\A\varphi|^{p}\dif x \leq \int_{(0,1)^{n}}F(z+\A\varphi)-F(z)\dif x
\end{align}
holds for all $z\in\mathscr{R}(\A)$ and $\varphi\in\hold_{c}^{1}((0,1)^{n};\R^{N})$; here, $\theta_{p}>0$ is the constant from Lemma~\ref{lem:compare}~\ref{item:compare1}. We apply Proposition~\ref{prop:Korn} to the particular choice $\psi(t):=t^{2}+t^{p}$, which is easily seen to satisfy the assumptions of Proposition~\ref{prop:Korn}. In the following, we hereafter denote $c>0$ the constant from \eqref{eq:modular} with this particular choice of $\psi$, $V=\R^{N}$ and $W=\mathscr{R}(\A)$. Then we have for all $\varphi\in\hold_{c}^{\infty}((0,1)^{n};\R^{N})$ and $z\in\R^{N\times n}$:
\begin{align*}
\int_{(0,1)^{n}}|\nabla\varphi|^{2}+|\nabla\varphi|^{p}\dif x \;\;\;\;\;\;&\!\!\!\!\!\!\!\!\stackrel{\text{Prop.}~\ref{prop:Korn}}{\leq} c \int_{(0,1)^{n}}|\A\varphi|^{2}+|\A\varphi|^{p}\dif x \\ 
& \!\!\!\!\!\!\!\!\!\!\stackrel{\mathrm{\ref{item:H3}},\,\eqref{eq:pstrongEQUIV1}}{\leq} \frac{2c\,\theta_{p}}{\ell}\int_{(0,1)^{n}}F(\pi_{\A}(z)+\A\varphi)-F(\pi_{\A}(z))\dif x \\ & \!\! = \frac{2c\,\theta_{p}}{\ell}\int_{(0,1)^{n}}G(z+\nabla\varphi)-G(z)\dif x, 
\end{align*}
since $\pi_{\A}(\nabla\varphi)=\A\varphi$ and $\pi_{\A}$ is linear. In conclusion, $G$ satisfies the hypotheses of Proposition~\ref{prop:CFM}~\ref{item:reductionreg1} with $\mu=\frac{\ell}{2c\theta_{p}}$. Hence all local minima of $w\mapsto \int G(\nabla w)\dif x$ are $\hold_{\locc}^{1,\alpha}$-partially regular. Since for any open $\omega\Subset\Omega$ and all $u\in\sobo^{1,p}(\Omega;\R^{N})$ there holds 
\begin{align*}
\mathscr{G}[u;\omega] := \int_{\omega}G(\nabla u)\dif x=\int_{\omega}F(\A u)\dif x = \mathscr{F}[u;\omega],
\end{align*}
any local minimiser of $\mathscr{F}$ is a local minimiser for $\mathscr{G}$ and so is $\hold_{\locc}^{1,\alpha}$-partially regular itself. The proof is complete. 
\end{proof}
\begin{proof}[Proof of Theorem~\ref{thm:main1}, case  $1<p<2$]
We proceed similarly as in the case $2\leq p<\infty$ and define $G\colon\R^{N\times n}\to\R$ by \eqref{eq:Gdef}. We now have to take care of the requisite quasiconvexity condition which here takes a slightly different form. Namely, by Lemma~\ref{lem:compare}~\ref{item:compare1} we need to establish that there exists $\nu>0$ such that 
\begin{align}\label{eq:SQCforG}
\int_{(0,1)^{n}}G(z+\nabla\varphi)-G(z)\dif x \geq \nu\int_{(0,1)^{n}}(1+|z|^{2}+|\nabla\varphi|^{2})^{\frac{p-2}{2}}|\nabla\varphi|^{2}\dif x
\end{align} 
holds for all $\varphi\in\hold_{c}^{1}((0,1)^{n};\R^{N})$ and $z\in\R^{N\times n}$. Note that, by definition of $G$ and the $p$-strong $\mathbb{A}$-quasiconvexity of $F$, \eqref{eq:SQCforG} will follow once we establish the existence of some $\mu>0$ such that 
\begin{align}\label{eq:KornPRmain}
\begin{split}
\int_{(0,1)^{n}}(1+|z|^{2}+&|\nabla\varphi|^{2})^{\frac{p-2}{2}} |\nabla\varphi|^{2}\dif x \\ & \leq \mu \int_{(0,1)^{n}}(1+|\pi_{\A}(z)|^{2}+|\A\varphi|^{2})^{\frac{p-2}{2}}|\A\varphi|^{2}\dif x
\end{split}
\end{align}
holds for all $\varphi\in\hold_{c}^{1}((0,1)^{n};\R^{N})$ and $z\in\R^{N\times n}$. Let $\varphi\in\hold_{c}^{\infty}((0,1)^{n};\R^{N})$ be arbitrary but fixed. Since $|\pi_{\A}(z)|\leq c|z|$ for some constant $c>0$, we utilise $p<2$ to obtain for all $x\in (0,1)^{n}$
\begin{align}\label{eq:pointwise}
\begin{split}
(1+|z|^{2}+|\nabla\varphi(x)|^{2})^{\frac{p-2}{2}}& |\nabla\varphi(x)|^{2} \\ & \leq \widetilde{c}(1+|\pi_{\A}(z)|^{2}+|\nabla\varphi(x)|^{2})^{\frac{p-2}{2}}|\nabla\varphi(x)|^{2}. 
\end{split}
\end{align}
Now define the auxiliary function $\Psi\colon[0,\infty)\to [0,\infty)$ given by 
\begin{align*}
\Psi(t):=(1+t)^{p-2}t^{2},\qquad t\geq 0
\end{align*}
as in Lemma~\ref{lem:compare}. Clearly, since $p>1$, $\Psi$ is of class $\Delta_{2}\cap\nabla_{2}$. By Lemma~\ref{lem:compare}\ref{item:compare2}, we may invoke \ref{item:F2} to obtain that the $\Delta_{2}$- and $\nabla_{2}$-constants of the shifted functions $\Psi_{a}$ (see \eqref{eq:shifted} for the definition) can be bounded independently of $a\geq 0$. Hence, by Proposition~\ref{prop:Korn} applied to $\Psi_{a}$ with $a:=|\pi_{\A}(z)|$ we obtain 
\begin{align}\label{eq:KornPR}
\int_{(0,1)^{n}}\Psi_{|\pi_{\A}(z)|}(|\nabla \varphi|)\dif x \leq c \int_{(0,1)^{n}}\Psi_{|\pi_{\A}(z)|}(|\A\varphi|)\dif x 
\end{align}
for all $\varphi\in\hold_{c}^{\infty}((0,1)^{n};\R^{N})$, where $c=c(\A,\Delta_{2}(\Psi),\nabla_{2}(\Psi))>0$. Now, we have
\begin{align}\label{eq:comparesubquadratic}
\Psi_{a}(t) \stackrel{\ref{item:F1}}{\simeq} \Psi''(a+t)t^{2} \stackrel{\eqref{eq:PsiBounds}_{1}}{\simeq} (1+a^{2}+t^{2})^{\frac{p-2}{2}}t^{2}
\end{align}
for all $t\geq 0$, where the constants implicit in '$\simeq$' do not depend on $a$. Therefore, 
\begin{align*}
\int_{(0,1)^{n}}(1+|z|^{2}+|\nabla\varphi|^{2})^{\frac{p-2}{2}}& |\nabla\varphi|^{2}\dif x  \\ & \stackrel{\eqref{eq:pointwise}}{\leq}\widetilde{c}\int_{(0,1)^{n}}(1+|\pi_{\A}(z)|^{2}+|\nabla\varphi|^{2})^{\frac{p-2}{2}}|\nabla\varphi|^{2}\dif x \\ 
& \stackrel{\eqref{eq:comparesubquadratic}}{\leq} \widetilde{\widetilde{c}} \int_{(0,1)^{n}}\Psi_{|\pi_{\A}(z)|}(|\nabla\varphi|)\dif x \\ 
& \stackrel{\eqref{eq:KornPR}}{\leq} \widetilde{C} \int_{(0,1)^{n}}\Psi_{|\pi_{\A}(z)|}(|\A\varphi|)\dif x \\
& \stackrel{\eqref{eq:comparesubquadratic}}{\leq} \widetilde{\widetilde{C}}\int_{(0,1)^{n}}(1+|\pi_{\A}(z)|^{2}+|\A\varphi|^{2})^{\frac{p-2}{2}}|\A\varphi|^{2}\dif x,
\end{align*}
where $\widetilde{\widetilde{C}}>0$ still does not depend on $z$ or $\pi_{\A}(z)$, respectively. This is \eqref{eq:KornPRmain} and yields that $G\colon\R^{N\times n}\to\R$ is $p$-strongly quasiconvex. As in the case $p\geq 2$, $G\in\hold^{2}(\R^{N})$ and obeys the growth bound $|G(z)|\leq c(1+|z|^{p})$ for all $z\in\R^{N\times n}$ and some $c>0$. Now we invoke Proposition~\ref{prop:CFM}~\ref{item:reductionreg2} to obtain that all local minima of $w\mapsto \int G(\nabla w)\dif x$ are partially $\hold_{\locc}^{1,\alpha}$-regular. As above in the case $p\geq 2$, this inherits to all local minima of $\mathscr{F}$, and the proof is complete. 
\end{proof} 
From here and based on \eqref{eq:dimnote}\emph{ff.}, it is straightforward to conclude Theorem~\ref{thm:main1} for general vector spaces $V$ and $W$; the elementary reduction scheme is explained in the Appendix, Section~\ref{sec:appendix}.
\subsection{Coerciveness and existence of minima}
Theorem~\ref{thm:main1} establishes the partial regularity of local minima of variational integrals \eqref{eq:functional}; we now briefly address their existence. For the sequel, denote following \cite{Gmeineder1a}
\begin{align*}
\sobo^{\A,p}(\Omega):=\{v\in\lebe^{p}(\Omega;V)\colon\;\A v\in\lebe^{p}(\Omega;W)\}
\end{align*}
and define $\sobo_{0}^{\A,p}(\Omega)$ as the closure of $\hold_{c}^{\infty}(\Omega;V)$ with respect to $\|u\|_{\A,p}:=\|u\|_{\lebe^{p}}+\|\A u\|_{\lebe^{p}}$. We now have the following
\begin{lemma}\label{lem:coercive}
Let $\A$ be an elliptic differential operator of the form \eqref{eq:form} and $\mathscr{A}$ an annihilator of $\A$. Moreover, let $F\in\hold(W)$ be a variational integrand that satisfies \emph{\ref{item:H2}} and \emph{\ref{item:H3}} for some $1<p<\infty$. Then for any open and bounded set $\Omega\subset\R^{n}$ and $u_{0}\in\sobo^{\A,p}(\Omega)$ the variational principle 
\begin{align}
\text{to minimise}\;\;\;\mathscr{F}[u;\Omega]:=\int_{\Omega}F(\A u)\dif x \qquad \text{over}\;u\in u_{0}+\sobo_{0}^{\A,p}(\Omega),
\end{align}
has a solution, and this solution is a local minimiser in the sense of \eqref{eq:local}.
\end{lemma} 
\begin{proof} 
By \ref{item:H3} and the equivalence of \eqref{eq:calAqc} and \eqref{eq:Aqc} by virtue of Lemma~\ref{lem:potentials}, we deduce that for all $\varphi\in\hold_{c}^{\infty}(\Omega;V)$ there holds 
\begin{align}\label{eq:lowerbound1}
F(0)\mathscr{L}^{n}(\Omega) + \ell\int_{\Omega}V_{p}(|\A\varphi|)\dif x \leq \int_{\Omega}F(\A\varphi)\dif x
\end{align}
with the function $V_{p}$ as in Lemma~\ref{lem:compare}. Now, since $F$ satisfies \ref{item:H3}, it is convex with respect to directions contained in the $\A$-rank-one cone $\mathscr{C}(\A)$ (cf.~\eqref{eq:puretensors}), which in turn spans $\mathscr{R}(\A)$. In combination with \ref{item:H2}, a straightforward adaptation of \cite[Lem.~5.5]{Giusti} thus yields that there exists a constant $c>0$ such that 
\begin{align}\label{eq:rankoneLipbound}
|F(w)-F(z)|\leq c(1+|w|^{p-1}+|z|^{p-1})|w-z| 
\end{align}
for all $w,z\in\mathscr{R}(\A)$. Note that $c>0$ only depends on the parameters implicit in \ref{item:H2} and \ref{item:H3}. By definition of $\sobo_{0}^{\A,p}(\Omega)$, this inequality directly yields that \eqref{eq:lowerbound1} holds for $\varphi\in\sobo_{0}^{\A,p}(\Omega)$, too. Since $|\cdot|^{p}-1\leq V_{p}(\cdot)$, for any $\varphi\in\sobo_{0}^{\A,p}(\Omega)$ we have
\begin{align}\label{eq:QCestimate1}
\begin{split}
\ell\int_{\Omega}|\A\varphi|^{p}-1\dif x & \leq \ell\int_{\Omega}V_{p}(\A\varphi)\dif x \\
& \!\!\!\!\stackrel{\eqref{eq:lowerbound1}}{\leq} \int_{\Omega}F(\A\varphi)\dif x - F(0)\mathscr{L}^{n}(\Omega)\\ 
& \leq \int_{\Omega}|F(\A(u_{0}+\varphi))-F(\A\varphi)|\dif x  \\ 
& + \int_{\Omega}F(\A(u_{0}+\varphi))\dif x - F(0)\mathscr{L}^{n}(\Omega)\\ 
& \!\!\!\!\stackrel{\eqref{eq:rankoneLipbound}}{\leq} c\int_{\Omega}(1+|\A\varphi|^{p-1}+|\A u_{0}|^{p-1})|\A u_{0}|\dif x  \\ & + \Big(\int_{\Omega}F(\A(u_{0}+\varphi))\dif x - F(0)\mathscr{L}^{n}(\Omega)\Big) =: \mathrm{I}+\mathrm{II}.  
\end{split}
\end{align}
At this stage, we employ Young's inequality to bound 
\begin{align}\label{eq:QCestimate2}
\mathrm{I} \leq \varepsilon \int_{\Omega}|\A\varphi|^{p}\dif x + c(\varepsilon)\int_{\Omega}|\A u_{0}|+|\A u_{0}|^{p}\dif x 
\end{align}
for $\varepsilon>0$, and consequently choose and fix $0<\varepsilon<\ell$ so that the first term on the right-hand side of \eqref{eq:QCestimate2} can be absorbed into the very left-hand side of \eqref{eq:QCestimate1}. As a consequence, there exists $c>0$ (which only depends on $\mathscr{L}^{n}(\Omega)$, $\A$ and the parameters underling hypotheses \ref{item:H2} and \ref{item:H3})  such that 
\begin{align}
\int_{\Omega}|\A u|^{p}\dif x \leq c\Big(\int_{\Omega}F(\A u)\dif x + \int_{\Omega}|\A u_{0}|+|\A u_{0}|^{p}\dif x+1\Big)
\end{align}
holds for all $u\in u_{0}+\sobo_{0}^{\A,p}(\Omega)$. In conclusion, $\mathscr{F}[-;\Omega]$ is bounded below on $u_{0}+\sobo_{0}^{\A,p}(\Omega)$. Since $\Omega$ is bounded, the usual Poincar\'{e} inequality and Proposition~\ref{prop:Korn} imply
\begin{align*}
\int_{\Omega}|u|^{p}\dif x & \leq c(p)\int_{\Omega}|u-u_{0}|^{p}\dif x + c(p)\int_{\Omega}|u_{0}|^{p} \\ 
& \leq c(\A,p,\Omega)\int_{\Omega}|\A (u-u_{0})|^{p} + c(p)\int_{\Omega}|u_{0}|^{p}. 
\end{align*} 
Hence, letting $(u_{j})\subset u_{0}+\sobo_{0}^{\A,p}(\Omega)$ be a minimising sequence for $\mathscr{F}[-;\Omega]$, we utilise $1<p<\infty$ and hereafter reflexivity of $\sobo^{\A,p}(\Omega)$ to find that a suitable non-relabelled subsequence converges weakly to some $u\in\sobo^{\A,p}(\Omega)$ in $\sobo^{\A,p}(\Omega)$. Now, since $u_{0}+\sobo_{0}^{\A,p}(\Omega)$ is a convex subset of the Banach space $\sobo^{\A,p}(\Omega)$ and, by definition of $\sobo_{0}^{\A,p}(\Omega)$, closed with respect to the norm topology on $\sobo^{\A,p}(\Omega)$, it is weakly closed. Therefore, $u\in u_{0}+\sobo_{0}^{\A,p}(\Omega)$. Finally, since $\A u_{j}\rightharpoonup \A u$ in $\lebe^{p}(\Omega;W)$, \ref{item:H1}--\ref{item:H3} imply by virtue of a straightforward higher order variant of \cite[Thm.~3.7]{FoMu} that 
\begin{align*}
\mathscr{F}[u;\Omega]\leq \liminf_{j\to\infty}\mathscr{F}[u_{j};\Omega]. = \inf_{u_{0}+\sobo_{0}^{\A,p}(\Omega)}\mathscr{F}[-;\Omega]. 
\end{align*}
Hence $u$ is a minimiser for $\mathscr{F}[-;\Omega]$, and the proof is complete. 
\end{proof}
In the gradient case, a slight variant of \ref{item:H3} is even equivalent to coerciveness, cf.~\textsc{Chen \& Kristensen} \cite{ChKr17}. We conclude this section with two remarks. 
\begin{remark}\label{rem:traces}
Note that, even for domains $\Omega\subset\R^{n}$ with smooth boundary $\partial\Omega$, elliptic operators $\A$ and $1<p<\infty$, there need not exist a boundary trace operator $\trace\colon\sobo^{\A,p}(\Omega)\to\lebe_{\locc}^{1}(\partial\Omega;V)$, cf.~\cite{BDG,Gmeineder1a}. Hence it is not directly possible to reduce the weak closedness of the Dirichlet classes in the proof of Lemma~\ref{lem:coercive} to the continuity properties of a respective trace operator. 
\end{remark}
\begin{remark}
Both Theorem~\ref{thm:main1} and Lemma~\ref{lem:coercive} exclude various constant rank operators $\mathscr{A}$ which do not have elliptic potentials, so e.g. $\mathscr{A}=\di$ (in which case $\A=\curl$). However, such operators do not give rise to partial regularity in the classical sense and, by direction~$\ref{item:THMpr}\Rightarrow\ref{item:THMelliptic}$ of Theorem~\ref{thm:main1}, the best to be expected is a $\hold^{0,\alpha}$-partial regularity result for $\A u$. Since only certain combinations of derivatives can potentially be proven to be partially $\hold^{0,\alpha}$-regular, this may be referred to a \emph{partial} partial regularity result, and we shall pursue this elsewhere. 
\end{remark} 
\section{Examples and extensions}\label{sec:examples}
\subsection{Examples}\label{sec:examples1}
To underline the applications of Theorem~\ref{thm:main1}, we explicitely address some examples of elliptic operators that frequently occur in applications; Theorem~\ref{thm:main1} or Theorems~\ref{thm:fullyautonomous} then provide the corresponding partial regularity results for the respective minimisers. 
\begin{enumerate}
\item\label{item:EX1} \emph{The symmetric gradient.} For $n\geq 1$, $V=\R^{n}$, $W=\rsym$, we put as in the introduction $\sg(u):=\frac{1}{2}(Du+Du^{\top})$. Setting $a\odot b:=\frac{1}{2}(a\otimes b + b\otimes a)$ for $a,b\in\R^{n}$, the elementary inequality 
\begin{align*}
\frac{1}{\sqrt{2}}|v|\,|\xi|\leq |v\odot\xi| = |\sg[\xi]v|\qquad\text{for all}\;v,\xi\in\R^{n}
\end{align*}
implies that $\sg$ is elliptic; also see \cite[Prop.~6.4]{Va11}.
\item\label{item:EX2} \emph{The trace-free symmetric gradient.} For $V=\R^{n}$, $W=\R_{\mathrm{sym,tf}}^{n\times n}:=\{z\in\rsym\colon\;\mathrm{tr}(z)=0\}$, we put $\sg^{D}(u):=\sg(u)-\frac{1}{n}\di(u)\mathbbm{1}_{n}$ with the $(n\times n)$-unit matrix $\mathbbm{1}_{n}\in\R^{n\times n}$. By \cite[Ex.~2.2(c)]{BDG}, $\sg^{D}$ is elliptic if and only if $n\geq 2$.  
\item\label{item:EX3} \emph{The exterior derivative.} For $\ell\in\{1,...,n-1\}$ and $V=\bigwedge^{\ell}\R^{n}$, $W=\bigwedge^{\ell+1}\R^{n}\times\bigwedge^{\ell-1}\R^{n}$, define $\A=(d,d^{*})$ as the first order differential operator whose symbol for $\xi\in\R^{n}$ and $v\in V$ is given by 
\begin{align*}
\A[\xi]v := (\xi\wedge v, *(\xi\wedge *v)). 
\end{align*}
By \cite[Prop.~6.6]{Va11}, this operator is elliptic, and so Theorem~\ref{thm:main1} complements the theme of partial regularity for differential forms as developed in \cite{BeckStroffolini13}.
\item\label{item:EX4} \emph{The div-curl-operator.} For $V=\R^{3}$ and $W=\R^{4}$, we define 
\begin{align*}
\A=\left(\begin{matrix} \di \\ \curl \end{matrix}\right)\colon u = (u_{1},u_{2},u_{3}) \mapsto \left(\begin{matrix} \partial_{1}u_{1}+\partial_{2}u_{2}+\partial_{3}u_{3} \\ \partial_{2}u_{3}-\partial_{3}u_{2} \\ \partial_{3}u_{1}-\partial_{1}u_{3} \\ \partial_{1}u_{2}-\partial_{2}u_{1}\end{matrix}\right). 
\end{align*}
It is then easy to see that $\A[\xi]v=0$ for $\xi\in\R^{3}\setminus\{0\}$ implies that $v=0$ and so $\A$ is elliptic. 
\end{enumerate}
Higher order elliptic operators, which Theorem~\ref{thm:higherorder} from below appeals to, can be canonically obtained by composing lower order elliptic differential operators. An example of different sort yet particular interest is given by 
\begin{itemize}
\item[(e)] \emph{The splitted Laplace-Beltrami operator.} Let $n\geq 2$, $\ell\in\{1,...,n-1\}$ and $V=\bigwedge^{\ell}\R^{n}$, $W:=\bigwedge^{\ell}\R^{n}\times\bigwedge^{\ell}\R^{n}$. The operator $\A u:=(dd^{*}u,d^{*}du)$ for $u\colon\R^{n}\to \bigwedge^{\ell}\R^{n}$ then is a second order differential operator on $\R^{n}$ from $V$ to $W$, and is elliptic because of $\Delta = dd^{*}+d^{*}d$ (also see \cite[Prop.~6.11]{Va11}).
\end{itemize}

\subsection{Extensions} 
We conclude the paper by discussing extensions of Theorem~\ref{thm:main1}, where we especially address  fully non-autonomous integrands or higher order differential operators, respectively. To emphasize the essentials, we now focus on the power growth case with $p\geq 2$.

These generalisations further manifest the metaprinciple that any regularity result being valid for $p$-strongly quasiconvex integrals also inherits to the situation of differential operators subject to the relevant strong $\mathscr{A}$-quasiconvexity condition.
\subsubsection{Fully non-autonomous integrands}\label{sec:fullynonaut}
Let $p\geq 2$ and $\A$ be an elliptic differential operator of the form \eqref{eq:form}; further, let $\mathscr{A}$ be an annihilator in the sense of \eqref{eq:complex}. Given an open and bounded set $\Omega\subset\R^{n}$,  we let $F\colon\Omega\times V\times \mathscr{R}(\A) \ni (x,y,z) \mapsto F(x,y,z)\in\R$ be a continuous integrand which satisfies the following set of hypotheses: 
\begin{enumerate}[label={(H\arabic{*}'')},start=1]
\item\label{item:H1'} $\mathrm{D}_{zz}^{2}F\in\hold(\Omega\times V\times\mathscr{R}(\A))$.
\item\label{item:H2'} There exists $c>0$ such that $|F(x,y,z)|\leq c(1+|z|^{p})$ holds for all $(x,y,z)\in \Omega\times V\times \mathscr{R}(\A)$. 
\item\label{item:H3'} There exist $c>0$, $0<\sigma<\frac{1}{p}$ and a bounded, concave and increasing function $\omega\colon\R_{\geq 0}\to\R_{\geq 0}$ such that $\omega(t)\leq t^{\sigma}$ for $t\geq 0$ and 
\begin{align*}
|F(x,y,z)-F(x',y',z)| \leq c(1+|z|^{p})\omega(|x-x'|^{p}+|y-y'|^{p})
\end{align*}
hold for all $x,x'\in\Omega$, $y,y'\in V$ and $z\in\mathscr{R}(\A)$. 
\item\label{item:H4'} There exists $\ell>0$ such that for every $(x,y,z)\in\Omega\times V \times\mathscr{R}(\A)$ and every $\varphi\in\hold_{c}^{1}((0,1))^{n};V)$ there holds 
\begin{align*}
\ell\int_{(0,1)^{n}}|\A\varphi|^{2}+|\A\varphi|^{p}\dif\eta \leq \int_{(0,1)^{n}}F(x,y,z+\A\varphi(\eta))-F(x,y,z)\dif\eta. 
\end{align*}
\item\label{item:H5'} $F$ can be uniformly minorised by a continuous function that is $p$-strongly $\mathscr{A}$-quasiconvex at zero: There exists $\widetilde{F}\in\hold(\mathscr{R}(\A))$ with 
\begin{align*}
\gamma\int_{(0,1)^{n}}|\A\varphi|^{p}\dif x \leq \int_{(0,1)^{n}}\widetilde{F}(\A\varphi)-\widetilde{F}(0)\dif x\qquad\text{for all}\;\varphi\in\hold_{c}^{1}((0,1)^{n};V)
\end{align*}
for some constant $\gamma>0$ such that $\widetilde{F}(z)\leq F(x,y,z)$ holds for all $(x,y,z)\in \Omega\times V\times\mathscr{R}(\A)$. 
\end{enumerate}
\begin{theorem}\label{thm:fullyautonomous}
Let $F\in\hold(\Omega\times V\times\mathscr{R}(\A))$ be a variational integrand satisfying \emph{\ref{item:H1'}--\ref{item:H5'}} from above. Then there exists $\beta_{0}\in (0,1)$ such that for every local minimiser $u\in\sobo_{\locc}^{1,p}(\Omega;V)$ of 
\begin{align*} 
\mathscr{F}[v;\omega]:=\int_{\omega}F(x,v,\A v)\dif x,\qquad\omega\subset\Omega,
\end{align*}
there exists an open set $\Omega_{0}\subset\Omega$ with $\mathscr{L}^{n}(\Omega\setminus\Omega_{0})=0$ such that $u$ is of class $\hold_{\locc}^{1,\beta_{0}}$ on $\Omega_{0}$. 
\end{theorem}
\begin{proof}
We argue as in the proof of Theorem~\ref{thm:main1}, case $p\geq 2$. Recalling that we may assume that $V=\R^{N}$ and $W=\mathscr{R}(\A)\subset\R^{N\times n}$, we define an integrand $G$ by 
\begin{align*}
G\colon \Omega\times \R^{N}\times\R^{N\times n}\ni (x,y,z) \mapsto F(x,y,\pi_{\A}(z)).
\end{align*}
Clearly, $G$ is still a continuous integrand with $\mathrm{D}_{zz}^{2}G\in\hold(\Omega\times\R^{N}\times\R^{N\times n})$ and $|G(x,y,z)|\leq c(1+|z|^{p})$ holds for all $(x,y,z)\in\Omega\times\R^{N}\times\R^{N\times n}$. Equally, we see that $G$ satisfies the corresponding analogues of \ref{item:H3'} and \ref{item:H4'}. On the other hand, put $\widetilde{G}(z):=\widetilde{F}(\pi_{\A}(z))$. Then $\widetilde{G}(z)\leq G(x,y,z)$ for all $(x,y,z)\in\Omega\times\R^{N}\times\R^{N\times n}$, and by the same argument as in the proof of Theorem~\ref{thm:main1}, case $p\geq 2$ (cf.~\eqref{eq:pstrongEQUIV1}), $\widetilde{G}$ satisfies 
\begin{align*}
\nu\int_{(0,1)^{n}}|D\varphi|^{p}\dif x \leq \int_{(0,1)^{n}}\widetilde{G}(D\varphi)-\widetilde{G}(0)\dif x
\end{align*}
for some $\nu>0$ and all $\varphi\in\hold_{c}^{1}((0,1)^{n};\R^{N})$. As a consequence, $G$ satisfies the requirements of \cite[Thm.~II.2]{AF1}. Hence there exists $\beta_{0}\in(0,1)$ such that for every local minimiser $u\in\sobo_{\locc}^{1,p}(\Omega;\R^{N})$ of $w\mapsto \int G(x,w(x),Dw(x))\dif x$ there exists an open set $\Omega_{0}\subset\Omega$ with $\mathscr{L}^{n}(\Omega\setminus\Omega_{0})=0$ and $u$ is of class $\hold_{\locc}^{1,\beta_{0}}$ in a neighbourhood of any of the points in $\Omega_{0}$. This inherits to the local minima of $\mathscr{F}$, and the proof is hereby complete.  
\end{proof}
\subsubsection{Higher order strong quasiconvexity}\label{sec:higherorder}
As alluded to in the introduction, if the annihilator $\mathscr{A}$ is given, then the potential $\A$ provided by Lemma~\ref{lem:potentials} does not need to have first order. We may, however, consider for $m\in\mathbb{N}$ an operator
\begin{align}\label{eq:form2}
\A := \sum_{|\alpha|=m}\A_{\alpha}\partial^{\alpha},
\end{align}
with $\A_{\alpha}\in\mathscr{L}(V;W)$, such that the symbol complex \eqref{eq:complex} is exact at $W$ for any $\xi\in\R^{n}\setminus\{0\}$. We then have the following higher order variant of Theorem~\ref{thm:main1}; as in the first order case, $\A$ is called \emph{elliptic} provided for each $\xi\in\R^{n}\setminus\{0\}$, $\A[\xi]=\sum_{|\alpha|=m}\xi^{\alpha}\A_{\alpha}\colon V\to W$ is injective.
\begin{theorem}\label{thm:higherorder}
Let $p\geq 2$ and let $\A$ be a constant rank differential operator of the form \eqref{eq:form2} of order $m\in\mathbb{N}$ and $\mathscr{A}$ of the form \eqref{eq:form1} such that the symbol complex \eqref{eq:complex} is exact at $W$ for any $\xi\in\R^{n}\setminus\{0\}$. Moreover, suppose that $F\colon W\to\R$ satisfies \ref{item:H1}--\ref{item:H3} and, for some constant $c>0$, the second derivative growth bound 
\begin{align}\label{eq:secondderivativeGB}
|F''(z)|\leq c(1+|z|^{p-2})\qquad\text{for all}\;z\in W. 
\end{align}
Then the following are equivalent:
\begin{enumerate}
\item $\A$ is elliptic. 
\item Every local minimiser of the integral functional $v\mapsto \int F(\A v)\dif x$ is $\hold_{\locc}^{m,\alpha}$-partially regular in the sense that there exists an open set $\Omega_{u}\subset\Omega$ with $\mathscr{L}^{n}(\Omega\setminus\Omega_{u})=0$ such that $u$ is of class $\hold_{\locc}^{m,\alpha}$ in an open neighbourhood of any of the points in $\Omega_{u}$ for any $0<\alpha<1$. 
\end{enumerate}
\end{theorem}
\begin{proof} 
As in the proof of Theorem~\ref{thm:main1}, ellipticity of $\A$ is easily seen to be necessary. On the other hand, Proposition~\ref{prop:Korn} yields for elliptic $\A$ of the form \eqref{eq:form2} that 
\begin{align*}
\int_{\R^{n}}\psi(|D^{m}v|)\dif x \leq c\int_{\R^{n}}\psi(|\A v|)\dif x\qquad\text{for all}\;v\in\hold_{c}^{\infty}(\R^{n};V).
\end{align*}
In fact, the previous inequality is derived completely analogously by replacing \eqref{eq:inverse} with 
\begin{align}\label{eq:inverse1}
\Phi_{\alpha,\A}(w)(x) = \mathscr{F}_{\xi\mapsto x}^{-1}\big[\xi^{\alpha}(\mathbb{A}^{*}[\xi]\mathbb{A}[\xi])^{-1}\mathbb{A}^{*}[\xi]\mathscr{F}w(\xi)\big],\qquad x\in\R^{n}
\end{align}
for $\alpha\in\mathbb{N}_{0}^{n}$ with $|\alpha|=m$. For our reduction procedure, we now note that in the higher order case as considered here, \eqref{eq:dimnote} takes the form $\dim(\mathscr{R}(\A))\leq \dim(\odot^{m}(\R^{n};\R^{N}))$ whereby we may assume that $\mathscr{R}(\A)\subset\odot^{m}(\R^{n};\R^{N})$. As is done for Theorem~\ref{thm:main1}, we may hereafter invoke the higher order partial regularity result due to \textsc{Kronz}  \cite[Thm.~2]{Kronz}; namely, if $G\in\hold^{2}(\odot^{m}(\R^{n};\R^{N}))$ is $p$-strongly quasiconvex in the sense that there exists $\nu>0$ such that 
\begin{align*}
\nu\int_{(0,1)^{n}}|D^{m}\varphi|^{2}+|D^{m}\varphi|^{p}\dif x \leq \int_{\Omega}G(z+D^{m}\varphi)-G(z)\dif x
\end{align*}
for all $\varphi\in\hold_{c}^{\infty}((0,1)^{n};\R^{N})$ and there exists $\lambda>0$ such that 
\begin{align*}
|G''(z)|\leq \lambda (1+|z|^{p-2})\qquad\text{for all}\;z\in\odot^{m}(\R^{n};\R^{N}), 
\end{align*}
then any local minimiser of the integral functional $u\mapsto \int G(D^{m}u)\dif x$ is $\hold_{\locc}^{m,\alpha}$-partially regular. Define $\pi_{\A}\colon\odot^{m}(\R^{n};\R^{N})\to\odot^{m}(\R^{n};\R^{N})$ with the obvious modifications in \eqref{eq:AOpdef}. To conclude the proof in analogy with that of Theorem~\ref{thm:main1}, we note that $G(z):=F(\pi_{\A}(z))$ moreover satisfies 
\begin{align*}
|G''(z)|\leq C(\A)|(\mathrm{D}_{zz}F)(\pi_{\A}(z))|\leq C(\A)(1+|\pi_{\A}(z)|^{p-2})\stackrel{p\geq 2}{\leq} C(\A)(1+|z|)^{p-2}
\end{align*}
for all $z\in\odot^{m}(\R^{n};\R^{N})$. Then the proof evolves as above for Theorem~\ref{thm:main1}. 
\end{proof}
We conclude the paper with a remark on the dimension reduction for the singular set. 
\begin{remark}[Partial regularity versus $\mathscr{H}^{s}$-bounds on the singular set]
Whereas the partial regularity of local minima within the framework of Theorems~\ref{thm:main1} and  \ref{thm:fullyautonomous} can be approached by reduction to the full gradient case, this is not so for Hausdorff dimension bounds of the singular set. By its nonlocality \cite{Kristensen}, quasiconvexity is substantially different from convexity. Moreover,  techniques as from the convex case \cite{CoFoIu,Giusti,Gm0,KristensenMingione,MingioneARMA1,Mingione1,Mingione2} do not apply here as they usually rely on the Euler-Lagrange system satisfied by the minimisers and positive definiteness of the integrands' second derivatives. To the best of our knowledge, in the quasiconvex case the only available bounds on the Hausdorff dimension of the singular set of local minima have been obtained by \textsc{Kristensen \& Mingione} \cite{KristensenMingione1} in the superquadratic case \emph{subject to a local Lipschitz assumption on the local minima}. Whereas this assumption seems plausible in the full gradient case, in the situation of Theorem~\ref{thm:main1} it would be more natural to suppose that $\A u\in\lebe_{\locc}^{\infty}(\Omega;W)$ for a given local minimiser $u\colon\Omega\to V$. Recalling the operators $\Phi_{j,\A}$ from \eqref{eq:inverse} (which can be realised as a local singular integral plus the identity), we find that $\Phi_{j,\A}\colon\lebe_{\locc}^{\infty}\to\mathrm{BMO}_{\locc}$. For minima of strongly quasiconvex integrals which are not locally Lipschitz but merely locally in $\mathrm{BMO}$, a dimension reduction for the singular set seems unavailable yet.

\end{remark}

\section{Appendix}\label{sec:appendix}
\begin{figure}
\begin{tikzpicture}[scale=1.5]
\draw[->] (-0.5,1) to (0.5,1); 
\node at (0,1.2) {$\zeta$};
\node at (-1,1) {$\mathscr{R}(\mathbb{A})$};
\node at (1,1) {$\mathscr{R}(\mathbb{A})$};
\node at (2,1) {$\oplus$};
\node at (3,1) {$\R^{l}$};
\node at (1,-0.5) {$\kappa(\mathscr{R}(\mathbb{A}))$};
\node at (2,-0.5) {$\oplus$};
\node at (3,-0.5) {$\kappa(\mathscr{R}(\mathbb{A}))^{\bot}$};
\draw[-] (2,0.75) to (2,0.5);
\draw[->] (2,0) to (2,-0.25);
\node at (2,0.25) {$\iota=(\kappa,\kappa')$};
\draw[dashed,<-] (0.5,-0.25) arc (200:150:35pt);
\node at (0.25,0.25) {$\kappa$};
\node at (3.65,0.28) {$\kappa'$};
\node at (3.25,0.25) {$\cong$};
\node at (0.62,0.25) {$\cong$};
\node at (4.25,-0.5) {$= \mathscr{L}(\R^{n};V)$};
\draw[->] (5,-0.475) -- (6,-0.475);
\node at (5.5,-0.3) {$\iota'$};
\node at (6.45,-0.45) {$\R^{N\times n}$};
\draw[dashed,<-] (3.375,-0.25) arc (-25:25:35pt);
\end{tikzpicture}
\caption{Identification procedure notation in the general case}
\end{figure}

To keep our main exposition simple and to elaborate on the essentials, up from \eqref{eq:dimnote} we adopted the viewpoint that $W=\mathscr{R}(\A)\subset\R^{N\times n}$. Here we give the elementary underlying framework that rigorously justifies this framework; to connect this setup with the usual one, we here adopt the viewpoint as in \cite{CFM98,Evans} that the gradient of a map $v\colon\R^{n}\supset\Omega\to\R^{N}$ is $\R^{N\times n}$-valued.

By \eqref{eq:dimnote}, there exists a linear injection $\kappa\colon \mathscr{R}(\A)\hookrightarrow \mathscr{L}(\R^{n};V)$. We then define $l:=\mathrm{codim}(\kappa(\mathscr{R}(\A)))$ to be the codimension of $\kappa(\mathscr{R}(\A))$ in $\mathscr{L}(\R^{n};V)$. Choose and fix an isomorphism $\kappa'\colon \R^{l}\to\kappa(\mathscr{R}(\A))^{\bot}$. Then 
\begin{align*}
\iota\colon \mathscr{R}(\A)\oplus \R^{l}\ni (z,z')\mapsto \kappa(z)+\kappa'(z')\in\mathscr{L}(\R^{n};V)
\end{align*}
is an isomorphism. We choose a particular orthonormal basis $\{\mathbf{v}_{1},...,\mathbf{v}_{N}\}$ for $V$ and denote $\lambda\colon\R^{N}\to V$ the corresponding coordinate map. Having fixed a particular basis for $V$, we now take the canonical identification $\iota'\colon \mathscr{L}(\R^{n};V)\stackrel{\cong}{\longrightarrow}\R^{N\times n}$. Put $\zeta\colon\mathscr{R}(\A)\ni z \mapsto (z,0)\in\mathscr{R}(\A)\oplus\R^{l}$. Then consider the differential operator 
\begin{align}
\widetilde{\A}v:=\sum_{j=1}^{n}(\iota'\circ\iota\circ \zeta\circ \A_{j}\circ\lambda)\partial_{j}v,\qquad v\colon\R^{n}\to \R^{N}, 
\end{align}
which is now a linear, homogeneous first order operator with constant coefficients on $\R^{n}$ from $\R^{N}$ to $\R^{N\times n}$. We may thus write for $v=(v_{1},...,v_{N})\colon\R^{n}\to\R^{N}$:
\begin{align}\label{eq:Atildedef}
\begin{split}
\widetilde{\A} v(x) = \left(\begin{matrix} \sum_{i=1}^{n}\sum_{k=1}^{N}a_{1,1}^{i,k}\partial_{i}v_{k}(x) & \hdots & \sum_{i=1}^{n}\sum_{k=1}^{N}a_{1,n}^{i,k}\partial_{i}v_{k}(x) \\ \vdots & \ddots & \vdots \\ \sum_{i=1}^{n}\sum_{k=1}^{N}a_{N,1}^{i,k}\partial_{i}v_{k}(x) & \hdots & \sum_{i=1}^{n}\sum_{k=1}^{N}a_{N,n}^{i,k}\partial_{i}v_{k}(x) \end{matrix}\right),
\end{split}
\end{align}
in turn being $\R^{N\times n}$-valued, for suitable $a_{\mu,\nu}^{i,k}\in\R$, $k,\mu\in\{1,...,N\}$ and $i,\nu\in\{1,...,n\}$. Now define a linear operator $\pi_{\widetilde{\A}}\colon\R^{N\times n}\to\mathscr{R}(\widetilde{\A})$ by 
\begin{align}\label{eq:piAdef}
\begin{split}
\pi_{\widetilde{\A}}\colon z & = (z_{ik})_{\substack{ i=1,...,n\\ k=1,...,N}} \\ & \mapsto \left(\begin{matrix} \sum_{i=1}^{n}\sum_{k=1}^{N}a_{1,1}^{i,k}z_{ik} & \hdots & \sum_{i=1}^{n}\sum_{k=1}^{N}a_{1,n}^{i,k}z_{ik} \\ \vdots & \ddots & \vdots \\ \sum_{i=1}^{n}\sum_{k=1}^{N}a_{N,1}^{i,k}z_{ik} & \hdots & \sum_{i=1}^{n}\sum_{k=1}^{N}a_{N,n}^{i,k}z_{ik}\end{matrix}\right).
\end{split}
\end{align}
\begin{lemma}\label{lem:reduction}
Given a differential operator $\A$ of the form \eqref{eq:form}, define $\widetilde{\A}$ by \eqref{eq:Atildedef}. Then the following hold: 
\begin{enumerate}
\item\label{item:red1} $\A$ is elliptic if and only if $\widetilde{\A}$ is. 
\item\label{item:red2} Let $F\in\hold(W)$ and let $\Omega\subset\R^{n}$ be open. Then a map $v\in\lebe_{\locc}^{p}(\Omega;V)$ with $\A v\in\lebe_{\locc}^{p}(\Omega;W)$ is a local minimiser for $\mathscr{F}[-;\Omega]$ if and only if $\lambda^{-1}v$ is a local minimiser for 
\begin{align*}
\widetilde{\mathscr{F}}[w;\omega]=\int_{\omega}F(\mathrm{proj}_{1}\circ\iota^{-1}\circ(\iota')^{-1}\circ \pi_{\widetilde{\A}}(\nabla w))\dif x,\qquad\omega\subset\Omega, 
\end{align*}
where $\mathrm{proj}_{1}\colon\mathscr{R}(\A)\oplus\R^{l}\ni (z,z') \mapsto z\in\mathscr{R}(\A)$ is the projection onto the first component. 
\end{enumerate} 
\end{lemma}
\begin{proof}
Ad~\ref{item:red1}. The symbol of $\widetilde{\A}$ is given by $\widetilde{\A}[\xi]v=\sum_{j=1}^{n}(\iota'\circ\iota\circ\zeta\circ\A_{j}\circ\lambda)\xi_{j}v$, $\xi\in\R^{n}$ and $v\in\R^{N}$. Suppose that $\widetilde{\A}[\xi]v=0$ for some $\xi\in\R^{n}\setminus\{0\}$. By bijectivity of $\iota'\circ\iota\colon \mathscr{R}(\A)\oplus\R^{l}\to\R^{N\times n}$, we deduce that $\zeta(\sum_{j=1}^{n}(\xi_{j}\A_{j}\circ\lambda)(v))=0$. This is only possible if $\sum_{j=1}^{n}\xi_{j}\A_{j}\lambda(v)=0$, and by ellipticity of $\A$, $\lambda(v)=0$. But $\lambda$ is an isomorphism, hence $v=0$. The other direction is similar. 

Ad~\ref{item:red2}. By \eqref{eq:Atildedef} and \eqref{eq:piAdef}, we have $\pi_{\widetilde{\A}}(\nabla (\lambda^{-1}v))=\widetilde{\A}(\lambda^{-1}v)$ and so, by definition of $\iota,\iota'$ and with $w=\lambda^{-1}v$, 
\begin{align}\label{eq:connections}
\begin{split}
\mathrm{proj}_{1}\circ\iota^{-1}\circ(\iota')^{-1}&\circ \pi_{\widetilde{\A}}(\nabla  w) \\ & = \mathrm{proj}_{1}\circ\iota^{-1}\circ (\iota')^{-1}\Big(\sum_{j=1}^{n}\iota'\circ\iota\circ\zeta\circ\A_{j}\circ\lambda (\partial_{j}w) \Big)\\
& = \sum_{j=1}^{n}\mathrm{proj}_{1}\circ\zeta\circ\A_{j}\circ\lambda(\partial_{j}w) \\ 
& = \sum_{j=1}^{n}\A_{j}\circ\lambda(\partial_{j}\lambda^{-1}v) = \A v. 
\end{split} 
\end{align}
Therefore, in particular, 
\begin{align*}
\widetilde{\mathscr{F}}[\lambda^{-1}v;\omega]= \int_{\omega}F(\A v)\dif x\qquad\text{for}\;\omega\subset\Omega, 
\end{align*}
which then yields the equivalence claimed in \ref{item:red2}. The proof is complete. 
\end{proof}
Instead of \eqref{eq:Gdef}, the proof of Theorem~\ref{thm:main1} is concluded by means of the integrand
\begin{align}\label{eq:G00def}
G\colon \R^{N\times n}\ni z \mapsto F(\mathrm{proj}_{1}\circ\iota^{-1}\circ(\iota')^{-1}\circ\pi_{\widetilde{\A}}(z)).
\end{align}

\end{document}